\documentclass[a4paper,twoside,12pt]{amsart}
\usepackage{amsthm, amsmath, amssymb}
\usepackage[margin=4cm]{geometry} 
\usepackage{xy} \xyoption{all}
\RequirePackage{mathrsfs}
\DeclareMathAlphabet{\othercal}{U}{eus}{m}{n}

\usepackage{MoritaE}

\title{Morita classes of microdifferential algebroids}

\author[A. D'Agnolo]{Andrea D'Agnolo}

\address{Dipartimento di Matematica Pura ed Applicata\\ 
Universit\`a di Padova\\ 
via Trieste 63\\
35121 Padova, Italy}

\email{dagnolo@math.unipd.it}

\author[P. Polesello]{Pietro Polesello} 

\address{Dipartimento di Matematica Pura ed Applicata\\ 
Universit\`a di Padova\\ 
via Trieste 63\\
35121 Padova, Italy}

\email{pietro@math.unipd.it}

\thanks{Partial support from the Fondazione Cariparo through the
project ``Differential methods in Arithmetic, Geometry and Algebra'', 
from Padova University through the grant CPDA122824/12, 
and from E.S.I. program ``Higher Structures in Mathematics and Physics" are acknowledged.
The first author wishes to thank R.I.M.S. of Kyoto University for
kind hospitality during the preparation of this paper.
The second author wishes to thank E.S.I. of Vienna and Paris 6 University for
kind hospitality during the preparation of this paper.}

\dedicatory{Dedicated to the memory of Louis Boutet de Monvel}   



\begin{document}

\maketitle

\begin{abstract} 
Projective cotangent bundles of complex manifolds are the local models of
complex contact manifolds. Such bundles are quantized by the algebra of
microdifferential operators (a localization of the algebra of differential
operators on the base manifold).

Kashiwara proved that any complex contact manifold $X$ is quantized by a
canonical microdifferential algebroid (a linear stack locally equivalent to an
algebra of microdifferential operators). Besides the canonical one, there can
be other microdifferential algebroids on $X$. Our aim is to classify them. More precisely,
let $Y$ be the symplectification of $X$. We prove that Morita (resp.
equivalence) classes of microdifferential algebroids on $X$ are described by
$H^2(Y;\C^\times)$. We also show that any linear stack locally equivalent to a
stack of microdifferential modules is in fact a stack of modules over a
microdifferential algebroid.

To obtain these results we use techniques of microlocal calculus, non-abelian cohomology and Morita theory for linear stacks.
\end{abstract}

\setcounter{tocdepth}{1}
\tableofcontents

\section*{Introduction}

In recent years there has been a lot of interest in the study of the deformation-quantizations of complex/algebraic symplectic manifolds, or more generally Poisson manifolds
(see for example \cite{Kon01,PS04,BGNT07,Pol08,CH11,BGNT11,KS12,Yek12}).
We are interested here in the study of (non-formal) quantizations of complex contact manifolds, which are locally modeled on the algebra of microdifferential operators of~\cite{S-K-K}, as initiated by Kashiwara~\cite{Kas96} (we refer to~\cite{PS04,KR08,DK11} for the relations between such quantizations and deformation-quantizations).

Let us describe such quantizations in a few more details.

Let $X$ be a complex contact manifold. By Darboux theorem, a local
model of $X$ is an open subset of the projective cotangent bundle $P^*M$
of a complex manifold $M$. Let $\she_{P^*M}$ be the sheaf of
microdifferential operators on $P^*M$ (a localization of the algebra of differential
operators on $M$). A microdifferential algebra ($\she$-algebra, for short) 
on
$X$ is a sheaf of $\C$-algebras locally isomorphic to $\she_{P^*M}$.

To quantize $X$ in the strict sense means to endow it with an $\she$-algebra.
(The relation between quantization and microdifferential operators is classical, see for example~\cite{Bou99,Pol15} for details.)
This might not be possible in general. 
However, Kashiwara~\cite{Kas96} proved that $X$ is endowed
with a canonical $\she$-algebroid $\stke_X$. This means the
following.
To an algebra $A$ one associates the linear category with one object
and elements of $A$ as its endomorphisms. Similarly, to a sheaf of algebras on $X$
one associates a linear stack. 
An $\she$-algebroid on $X$ is a $\C$-linear stack locally equivalent to one associated with an
$\she$-algebra. 
In this sense, a quantization of $X$ is the datum of an $\she$-algebroid on $X$.

Having to deal with an algebroid instead of an algebra is not very restrictive. 
For example, it makes sense to consider coherent modules over an $\she$-algebroid, and in particular regular modules along complex involutive subvarieties of $X$. The Lagrangian case is of particular interest, since these modules are the counterpart of microlocal perverse sheaves in the Riemann-Hilbert correspondence (see \cite{Kas96,Was05,DP05,DS07}).
A very interesting example of non-coherent module is the (twisted)
sheaf of microfunctions along a totally real, I-symplectic Lagrangian submanifold (see~\cite[\S4]{Kas96}).

A natural problem is to classify $\she$-algebroids on $X$.

The canonical $\she$-algebroid $\stke_X$ is endowed with an anti-involution,
corresponding to the operation of taking the formal adjoint of microdifferential operators. It is also endowed with a natural order filtration, and its associated graded algebroid is trivial. 
It is shown in~\cite{Pol07} that $\stke_X$ is unique among such $\she$-algebroids. 

Equivalence classes of filtered $\she$-algebroids with trivial graded (but not necessarily endowed with an anti-involution) are classified in~\cite{Pol15}.

In this paper, we classify arbitrary $\she$-algebroids on $X$.
These include in particular filtered 
$\she$-algebroids whose associated graded is non-trivial
(what would be called twisted quantizations of $X$ in the terminology of~\cite{BGNT07,Yek12}).

We also classify stacks of twisted $\she$-modules, i.e.~stacks locally equivalent to the stack of modules over an $\she$-algebra.
These are the most general stacks where the notion of microdifferential modules makes sense. 
Roughly, they are obtained by patching sheaves of $\she$-algebras  through Morita equivalences.

More precisely, in Theorems~\ref{thm:Emoritatrivial} and~\ref{thm:class}, and Corollary~\ref{cor:class}, we prove the following results:
\begin{itemize}
\item[(i)]
Two $\she$-algebroids are equivalent if and only
if they are Morita equivalent, i.e.~their stacks of modules are equivalent.
\item[(ii)]
Any stack of twisted $\she$-modules is globally equivalent to the stack of modules
over an $\she$-algebroid.
\item[(iii)]
The set of equivalence classes (resp. Morita classes) of $\she$-algebroids is canonically isomorphic to $H^2(Y;\C^\times_Y)$, for $Y$ the symplectification of $X$.
\item[(iv)]
The group of invertible $\stke_X$-bimodules is isomorphic to $H^1(Y;\C^\times_Y)$.
\end{itemize} 

We also give explicit realizations of the isomorphisms in (iii) and (iv).

\smallskip

To obtain our results, we use techniques  of microlocal calculus, non-abelian cohomology and  Morita theory for linear stacks.

\smallskip

The content of this paper is as follows.

In Section~\ref{se:nonab} we collect, without proofs, the main facts of non-abelian cohomology we need to prove our results. Note that cohomology with values in non-abelian groups was already used in~\cite{Bou99} to classify $\she$-algebras, and cohomology with values in 2-groups is used in~\cite{Pol08,Pol15} for the classification of algebroids.

In Section~\ref{se:alg} we give the basics of the theory of algebroids.

In Section~\ref{se:Mor},  we detail Morita theory for linear stacks. In particular, the notion of Picard good stacks allows us to recover results of~\cite{Low08}. Morita theory for linear categories is developed in~\cite{Mit65,Mit85}. The case of stacks of modules over sheaves of algebras is discussed in~\cite{KS06} (see also~\cite{DP04}).

In Section~\ref{se:mic} we recall some results from the theory of microdifferential operators. In particular, we provide a detailed proof of an unpublished result, due to Kashiwara, on the structure of invertible bimodules (see Theorem~\ref{thm:Eeloc} below). This allows us to prove the key Theorem~\ref{thm:localEPicardGood}.

In Section~\ref{se:results} we prove our main classification results, as (i)--(iv) above.

In Appendix~\ref{se:coc} we recall the cocycle description of algebroids and functors between them.

\smallskip

For symplectic manifolds, or more generally for Poisson manifolds, some results related to ours appeared in the literature: on a complex symplectic manifold, deformation quantization algebroids with anti-involution and trivial graded  have been classified up to equivalence in~\cite{Pol08} (see also \cite{BGNT07,BGNT11} for the possibly twisted case), whereas Morita-type results for deformation quantization algebras are obtained in~\cite{Bur02,BW04,BDW09} for real Poisson manifolds, and in~\cite{Yek10} in the algebraic setting.

\subsection*{Acknowledgments}
We express our gratitude to Masaki Kashiwara for useful discussions and
for communicating Theorem~\ref{thm:Eeloc} to us.
We would like to thank the referee for pointing out a mistake in a previous version of this paper.

\subsection*{Convention} In this text, when dealing with categories and stacks, we will
not mention any smallness condition (with respect to a given universe), leaving to the
reader the task to make it precise when necessary.

\section{Non-abelian cohomology}\label{se:nonab}

We are interested in classifying $\she$-algebroids and stacks of twisted $\she$-modules.
Thanks to the existence of a canonical $\she$-algebroid, 
this amounts to classify stacks locally
equivalent to a given one. To this end, we recall here some
techniques of cohomology with values in a stack of
$2$-groups. 

The classical reference for stacks is~\cite{Gir71}. 
A quick introduction can be found in~\cite{DP04} and~\cite[chapter~19]{KS06}.

For stacks of $2$-groups we refer to~\cite{Bre94} (where the term $gr$-stack is used instead) 
and to~\cite[\S1.4]{Del73} for the strictly commutative case. See also \cite{AN09} for an explicit description in terms of crossed modules. We follow the presentation of~\cite{Pol08}.

\medskip
Let $X$ be a topological space (or a site).

\subsection{Stacks}\label{se:linearstacks}
A prestack $\stkc$ on $X$ is a lax presheaf of categories. 
It is lax in the sense that for a chain of three open
subsets $W\subset V\subset U$ the restriction functor $\cdot |_W \colon \stkc(U)\to \stkc(W)$
coincides with the composition $\stkc(U)\to[\cdot |_V] \stkc(V)\to[\cdot |_W] \stkc(W)$ only up to
an invertible transformation (such transformations need to satisfy a natural cocycle
condition for chains of four open subsets). 

For $\objc,\objc'\in\stkc(U)$, denote by $\shHom[\stkc](\objc,\objc')$
the presheaf on $U$ given by $U \supset V\mapsto
\Hom[\stkc(V)](\objc|_V,\objc'|_V)$. 
One says that $\stkc$ is a separated prestack if $\shHom[\stkc](\objc,\objc')$
is a sheaf for any $\objc,\objc'$.
A stack on $X$ is a separated prestack satisfying a natural descent
condition, analogous to that for sheaves.

Given a stack $\stkc$, we denote by $\pi_0(\stkc)$ the sheaf associated to the presheaf
 $X \supset U\mapsto \{\text{isomorphism classes of objects in }\stkc(U)\}$, and by 
$\stkFun(\stkc,\stkc')$, for another stack $\stkc'$, the stack whose objects on $U\subset X$ are functors from $\stkc|_U$ to $\stkc'|_U$ and whose morphisms are transformations.

Let $\varphi\colon Y \to X$ be a continuous map (or a morphism of sites). For $\stkd$ a stack 
on $Y$ and $\stkc$ a stack on $X$, we denote by $\oim \varphi \stkd$ and $\opb \varphi \stkc$ the 
stack-theoretical direct and inverse image, respectively. 
Recall that $\opb \varphi \stkc$ is the stack on $Y$ associated to the separated prestack 
$\varphi^{+}\stkc$, defined on an open subset $V \subset Y$ by the category
\begin{align*}
\operatorname{Ob}(\varphi^{+}\stkc (V)) &= \DUnion_{U\supset \varphi(V) \atop U \text{
open} }\operatorname{Ob}(\stkc (U)),\\
\Hom[\varphi^{+}\stkc (V)](\objc_{U},\objc_{U'}) &= \sect(V;
\opb \varphi \shHom[\stkc](\objc_U|_{U \cap U'},\objc_{U'}|_{U \cap U'})).
\end{align*}
One checks that there is a natural equivalence (in fact, a $2$-adjunction)
\begin{equation}\label{eq:2adj}
\oim \varphi \stkFun(\opb \varphi \stkc,\stkd) \equi[]
\stkFun(\stkc,\oim \varphi \stkd).
\end{equation}
Hence there are adjunction functors
$$
\stkc \to \oim \varphi \opb \varphi \stkc, \qquad \opb \varphi \oim \varphi \stkd \to \stkd.
$$
By using the first adjunction, one gets an isomorphism of sheaves
\begin{equation}\label{eq:inv-pi_0}
\opb \varphi \pi_0(\stkc)\isoto  \pi_0(\opb \varphi\stkc).
\end{equation}

\subsection{Stacks of $2$-groups}\label{se:2-groups}

Let $\stkc$ be a stack on $X$. Denote by $\stkAut(\stkc)$ the substack of $\stkFun(\stkc,\stkc)$ whose objects are the auto-equivalences of $\stkc$, and whose morphisms are the {\em invertible} transformations.
For $\shu=\{U_{i}\}_{i\in I}$ an open cover of $X$, set
\[
U_{ij} = U_{i} \cap U_{j},\quad U_{ijk} = U_{i} \cap U_{j} \cap
U_{k},\quad\text{etc.}
\]
Proposition~\ref{pr:patch} for
$\stkc_i=\stkc'_i=\stkc|_{U_i}$ describes how to patch objects and morphisms
of $\stkAut(\stkc)$. With notations as in~Proposition~\ref{pr:patch}, let $H^1(\shu;
\stkAut(\stkc))$ be the pointed set of equivalence classes of pairs
$\pair{\functf_{ij}}{\transfa_{ijk}}{ijk\in I}$ satisfying the cocycle
condition~\eqref{eq:alpha}, modulo the coboundary relation given by 
pairs $\pair{\functg_i}{\transfb_{ij}}{ij\in J} $ satisfying~\eqref{eq:alpha2}. 
As we recall in Remark~\ref{rk:ref}, an open cover $\shv$ finer than $\shu$ induces a well defined map
$H^1(\shu; \stkAut(\stkc)) \to H^1(\shv; \stkAut(\stkc))$.
Hence one sets
\begin{equation}\label{eq:cohAutA}
H^1(X; \stkAut(\stkc)) = \ilim[\shu]H^1(\shu; \stkAut(\stkc)).
\end{equation}

By Proposition~\ref{pr:patch} it follows immediately
\begin{corollary}\label{cor:h1aut}
The pointed set $H^1(X; \stkAut(\stkc))$ is isomorphic to the pointed set of equivalence classes of stacks locally equivalent to $\stkc$.
\end{corollary}

Let us recall how to make the construction~\eqref{eq:cohAutA} functorial.

\medskip

A $2$-group (also called a \emph{gr}-category in~\cite{Bre94,AN09}) is a category endowed with a group structure both on morphisms and on objects. More precisely, a category $\catg$ is a $2$-group if it is a groupoid  (i.e.~all morphisms are invertible) and it has a structure
$(\catg,\otimes,\mathbf{1})$ of monoidal category such that each object $\objc\in\catg$ has a chosen (right) inverse.
Functors of $2$-groups and transformations between them are monoidal functors and monoidal transformations, respectively.

A stack $\stkg$ is called a stack of $2$-groups (a \emph{gr}-stack in~\cite{Bre94,AN09}) if its categories of sections are $2$-groups, its restrictions are functors of $2$-groups and its
transformations between restriction functors are monoidal. Functors of stacks
of $2$-groups are functors of monoidal stacks.

Recall that one sets $\pi_1(\stkg) = \shHom[\stkg](\mathbf{1},\mathbf{1})$. Both $\pi_0(\stkg)$ and 
$\pi_1(\stkg)$ are sheaves of groups, the latter being necessarily commutative. 
Any functor of stacks of $2$-groups induces a group homomorphism at the level of $\pi_1$ and $\pi_0$.

\begin{example}\label{exa:g01}
For $\shg$ a sheaf of groups, denote by $\shg[0]$ the stack obtained by
trivially enriching $\shg$ with identity arrows, and by $\shg[1]$ the stack of right
$\shg$-torsors. Then $\shg[0]$ is a stack of $2$-groups, and $\shg[1]$ is a
stack of $2$-groups if and only if $\shg$ is commutative.
\end{example}

Another example of stack of $2$-groups is given by $\stkAut(\stkc)$ for a stack $\stkc$.
Let $\stkg$ be a stack of $2$-groups and $\shu$ an open cover of $X$. 
One can extend as follows the construction~\eqref{eq:cohAutA}, where one should 
read ``$\otimes$'' instead of ``$\circ$'' in all diagrams in Appendix~\ref{se:cocy}.

A $1$-cocycle with values in $\stkg$ is a pair $\pair{\objc_{ij}}{a_{ijk}}{ijk\in I}$ with
$\objc_{ij}\in\stkg(U_{ij})$ and $a_{ijk}\in
\Hom[\stkg(U_{ijk})](\objc_{ik},\objc_{ij}\tens \objc_{jk})$ satisfying
\eqref{eq:alpha}. Two such 1-cocycles $\pair{\objc_{ij}}{a_{ijk}}{ijk\in
I}$ and $\pair{\objc'_{ij}}{a'_{ijk}}{ijk\in I}$ are cohomologous if
there is a pair $\pair{\objd_i}{b_{ij}}{ij\in I}$ with
$\objd_i\in\stkg(U_i)$ and $b_{ij}\in\Hom[\stkg(U_{ij})]( \objc'_{ij}\tens \objd_{j}, \objd_{i}\tens \objc_{ij})$
satisfying \eqref{eq:alpha2}.

The first cohomology pointed set of $\stkg$ on $X$ is given by
\[
H^1(X; \stkg) = \ilim[\shu]H^1(\shu; \stkg),
\]
where $H^1(\shu; \stkg)$ denotes the pointed set of equivalence classes of
$1$-cocycles on $\shu$, modulo the relation of being cohomologous.
One can also define the cohomology in degree 0 and $-1$.
This construction is functorial in the sense that short exact
sequences of 2-groups induce long exact cohomology sequences (in a sense to be made
precise). In particular, equivalent 2-groups have isomorphic cohomologies.

With notation as in Example~\ref{exa:g01} one has
\begin{equation}\label{eq:HiG}
H^1(X; \shg[i]) \simeq H^{1+i}(X; \shg)
\quad\text{for }i=0,1,
\end{equation}
where $\shg$ is assumed to be abelian if $i=1$.
Here, the pointed set $H^1(X; \shg)$ is defined by Cech cohomology and $H^2(X;
\shg)$ is computed using hypercoverings.

\subsection{Crossed modules}\label{se:crossed-mod}
A crossed module is the data
\[
\shg^\bullet = (\shg^{-1} \to[d] \shg^0,\ \delta)
\]
of a complex of sheaves of groups and of a left action $\delta$ of $\shg^0$ on
$\shg^{-1}$ such that for any $f\in\shg^0$ and $a\in\shg^{-1}$
\[
d \circ \delta(f) =\ad(f) \circ d, \qquad
\delta\bigl(d(a)\bigr)=\ad(a),
\]
where $\ad(a)(b) = a b a^{-1}$. A morphism of crossed modules is a morphism
of complexes of sheaves of groups compatible with the left actions.

There is a functorial way of associating to a crossed module a stack of
$2$-groups as follows. For $\shg^\bullet$ a crossed module one denotes by $[\shg^\bullet]$ the
stack of $2$-groups associated to the separated prestack whose objects on $U\subset X$
are sections $f\in\shg^{0}(U)$ and whose morphisms $f\to f'$ are sections
$a\in \shg^{-1}(U)$ satisfying $f' = d(a) f$. Then $[\shg^\bullet]$ is a stack of
$2$-groups, with monoidal structure given by $f\tens g=fg$ at the level of
objects and by $a\tens b = a\delta(f)(b)$ at the level of morphisms, for
$a\colon f\to f'$ and $b\colon g\to g'$.

One checks that there are isomorphisms of groups
$$
\pi_i([\shg^\bullet])\simeq H^{-i}(\shg^\bullet), \qquad i=0,\,1
$$
and, with notation and conventions as in Example~\ref{exa:g01}, equivalences 
of stacks of 2-groups
\[
\bigl[\shg[i]\bigr]\equi[]\shg[i], \qquad i=0,\,1,
\]
where the structure of crossed module on the complex $\shg[i]$ (obtained by placing $\shg$ in $-i$ position) is the trivial one.

\subsection{Strictly abelian crossed modules}
Denote by $\catd^{[-1,0]}(\Z_X)$ the full subcategory of the
derived category of sheaves of abelian groups whose
objects have cohomology concentrated in degree $-1$ and $0$. Consider a complex of
abelian groups $\shf^\bullet \in \catc^{[-1,0]}(\Z_X)$ as a crossed module with
trivial left action. Then the functor $\shf^\bullet \mapsto [\shf^\bullet]$
factorizes through $\catd^{[-1,0]}(\Z_X)$ and one has
\begin{equation}\label{eq:h1commcross}
H^1(X;\shf^\bullet) = H^1(X;[\shf^\bullet]).
\end{equation}

Let $\psi\colon X \to Y$ be a continuous map (or a morphism of sites). 
The inverse and direct image of stacks of $2$-groups are again stacks of $2$-groups, and one has
\begin{equation}\label{eq:roim}
\opb\psi[\shg^\bullet] \equi[] [\opb\psi\shg^\bullet] \qquad 
\oim\psi[\shf^\bullet] \equi[] [\tau_{\leq 0}\roim\psi\shf^\bullet],
\end{equation}
where $\tau_{\leq 0}$ is the truncation functor. In particular, for a sheaf of abelian groups
$\shf$, one gets
\begin{equation}\label{eq:pi-roim}
\pi_i(\oim\psi (\shf[1]))\simeq R^{1-i}\oim\psi\shf,  \qquad i=0,\,1.
\end{equation}

\section{Algebroids}\label{se:alg}

Mitchell~\cite{Mit72} showed how algebras can be replaced by linear categories.
Similarly, sheaves of algebras can be replaced by linear stacks.
An algebroid is a linear stack locally equivalent to an algebra.
This notion, already implicit in~\cite{Kas96}, was introduced in~\cite{Kon01} 
and developed in~\cite{DP05} (see also \cite[\S2.1]{KS12} and \cite{DK11}).
It is the linear analogue of the notion of gerbe from~\cite{Gir71}: an
algebroid is to a gerbe as an algebra is to a group.

\medskip

Let $X$ be a topological space (or a site), and $\shr$ a sheaf of commutative
rings on $X$.

\subsection{Linear stacks}\label{sse:linearstacks}
A stack $\stkc$ on $X$ is called $\shr$-linear ($\shr$-stack, for short) if for any 
$\objc,\objc'\in\stkc(U)$ the sheaf $\shHom[\stkc](\objc,\objc')$ is endowed with an 
$\shr|_U$-module structure compatible with composition. In particular, 
$\shEnd[\stkc](\objc)$ has an $\shr|_U$-algebra structure with product given 
by composition. A functor between $\shr$-linear stacks is called $\shr$-linear 
($\shr$-functor, for short) if it is $\shr$-linear at the level of morphisms, while no 
linearity conditions are required on transformations.

One says that two $\shr$-stacks are equivalent if they
are equivalent through an $\shr$-functor. This implies that the
quasi-inverse is also an $\shr$-functor. We denote by $\equi$ this
equivalence relation.

The center $Z(\stkc)$ of an $\shr$-stack $\stkc$ is the sheaf
of endo-transformations of the identity functor $\stack \id_\stkc$. 
It has a natural structure of sheaf of commutative $\shr$-algebras. 
Note that a stack $\stkc$ is $\shr$-linear if and only if it is $\Z$-linear 
and its center is an $\shr$-algebra.

If $\stkc$ is an $\shr$-stack, then its opposite stack $\stkc^\op$ is again $\shr$-linear.
For $\stkd$ another $\shr$-stack, $\stkFun[\shr](\stkc,\stkd)$ denotes the full substack of 
$\stkFun(\stkc,\stkd)$ whose objects are $\shr$-functors, and is itself an $\shr$-stack. 
The tensor product $\stkc\tens[\shr]\stkd$ is the $\shr$-stack associated with the prestack
$U\mapsto \stkc(U)\tens[\shr(U)]\stkd(U)$ whose objects are pairs in
$\stkc(U)\times\stkd(U)$, with morphisms
\[
\Hom[{\stkc(U)\tens[\shr(U)]\stkd(U)}]((\objc,\objd),(\objc',\objd'))
= \Hom[\stkc(U)](\objc,\objc')
\tens[\shr(U)]
\Hom[\stkd(U)](\objd,\objd').
\]

\begin{lemma}\label{lem:stktensadj}
If $\shr$ is an $\shs$-algebra and $\stke$ an $\shs$-stack, then
\[
\stkFun[\shs](\stkc\tens[\shr]\stkd,\stke) \equi
\stkFun[\shr](\stkc,\stkFun[\shs](\stkd,\stke)).
\]
(This is in fact a $2$-adjunction.)
\end{lemma}
Let $\varphi\colon Y \to X$ be a continuous map (or a morphism of sites). Then 
$\opb \varphi\stkc$ is $\opb \varphi \shr$-linear and there is a $\opb \varphi \shr$-equivalence
$$
\opb \varphi (\stkc\tens[\shr]\stkd)\equi[] \opb \varphi \stkc\tens[\opb \varphi\shr]\opb \varphi\stkd.
$$
If $\stke$ is a $\opb \varphi \shr$-stack, then $\oim \varphi \stke$ is $\shr$-linear and there is an $\shr$-functor
\begin{equation}
\label{eq:dir-im}
\oim \varphi \stke\tens[\shr]\oim \varphi \stkf \to  \oim \varphi (\stke\tens[ \opb \varphi \shr]\stkf).
\end{equation}

\subsection{Modules over a linear stack}

Denote by $\catMod(\shr)$ the category of $\shr$-modules and by
$\stkMod(\shr)$ the corresponding $\shr$-stack given by $U\mapsto \catMod(\shr|_U)$.

For $\stkc$ an $\shr$-stack, the $\shr$-stack of (left) $\stkc$-modules is defined by
\begin{equation}
\label{eq:Mod}
\stkMod(\stkc) = \stkFun[\shr](\stkc,\stkMod(\shr)).
\end{equation}
(It follows from Lemma~\ref{lem:ModR} that this definition does not
depend on the base ring.)

The contravariant $2$-functor $\stkMod(\dummy)$ is defined as follows.
On objects, it is given by \eqref{eq:Mod}.
Consider the diagram
\[
\xymatrix@C=3em{ \stkc \ar@/_.8em/[r]_{\functf'}
\ar@/^.8em/[r]_{\Downarrow\, \transfd}^{\functf} & 
\stkd \ar[r]^-{\shn} & \stkMod(\shr),}
\]
where $\functf, \functf'$ and $\shn$ are $\shr$-functors, and $\transfd$ a transformation.
To the $\shr$-functor $\functf\colon \stkc\to\stkd$ one associates the $\shr$-functor
\[
\stkMod(\functf)\colon \stkMod(\stkd)\to\stkMod(\stkc), \qquad
\shn \mapsto \shn \circ \functf,
\]
and to the transformation $\transfd\colon \functf\to\functf'$ one associates the
transformation,
\[
\stkMod(\transfd)\colon\stkMod(\functf)\to\stkMod(\functf'),
\]
such that $\stkMod(\transfd)(\shn) = \id_\shn \mathop\bullet \transfd$
is the morphism associated to $\shn\in\stkMod(\stkd)$,
where $\bullet$ denotes the horizontal composition of transformations. In
other words, for $\objc\in\stkc$ one has 
$\stkMod(\transfd)(\shn)(\objc) = \shn(\transfd(\objc))$ as morphisms from
$\shn(\functf(\objc))$ to $\shn(\functf'(\objc))$ in $\stkMod(\shr)$. We use
the notation
\begin{equation}\label{eq:2mod}
{}_\functf(\dummy) = \stkMod(\functf).
\end{equation}

\subsection{Algebras as stacks}

Let $\sha$ be a sheaf of $\shr$-algebras. 
Denote by $\sha^\op$ the opposite algebra, by $\stkMod(\sha)$ the 
$\shr$-stack of left $\sha$-modules and by
$\hom[\sha](\cdot,\cdot) = \hom[\stkMod(\sha)](\cdot,\cdot)$
the internal hom-functor.

Denote by $\astk\sha$ the full substack of $\stkMod(\sha^\op)$
whose objects are locally free right $\sha$-modules of rank one. For any
$\shn\in\astk\sha(U)$ there is an $\shr|_U$-algebra isomorphism
$\shEnd[\astk \sha](\shn) \simeq \sha|_U$. 
Note that the stack $\astk\sha$ has a canonical global object given by $\sha$ itself with its
structure of right $\sha$-module. In particular, the sheaf $\pi_0(\astk \sha)$ is the singleton-valued constant sheaf.

For $f\colon\sha\to\shb$ an $\shr$-algebra morphism, denote by 
$\astk f\colon \astk\sha\to\astk\shb$ the
$\shr$-functor induced by the extension of scalars $(\dummy)\tens[\sha]\shb$.
We thus have a functor between the stack of $\shr$-algebras and that of
$\shr$-stacks
\begin{equation*}
\astk{(\dummy)} \colon \shr\text-\stack{Alg}_X \to \shr\text-\stack{Stk}_X.
\end{equation*}

\begin{remark}\label{rm:aprstk}
Let $\aprstk\sha$ be the separated prestack $U\mapsto \astk{\sha(U)}$, where $ \astk{\sha(U)}$ denotes the $\shr(U)$-category with one object and sections of $\sha(U)$ as its endomorphisms.
By Yoneda lemma (see \S\ref{sse:Y}),  the stack associated to $\aprstk\sha$ is $\shr$-equivalent to $\astk\sha$.
\end{remark}

The stack $\shr\text-\stack{Stk}_X$ is naturally upgraded to a 2-stack
by considering transformations of functors. 
Note that for any two $\shr$-algebras $\sha$ and $\shb$ on $U\subset X$ and any
$\shr$-functor $\functf\colon\astk\sha\to\astk\shb$ there exist a cover
$\shu=\{U_{i}\}_{i\in I}$ of $U$ and morphisms of $\shr$-algebras $f_i\colon
\sha|_{U_i}\to\shb|_{U_i}$ such that $\functf|_{U_i} \simeq \astk f_i$.

\begin{definition}
One says that an $\shr$-stack $\stkc$ is equivalent to an $\shr$-algebra $\sha$ if
$\stkc\equi\astk\sha$.
\end{definition}

In Proposition~\ref{pro:AequiB} we characterize the condition of
equivalence between algebras.

Recall that a stack $\stkc$ is non-empty if it has at least one global object, and it is locally connected by isomorphisms if any two objects $\objc,\objc'\in\stkc(U)$ are locally isomorphic. 
If $\stkc$ is $\shr$-linear, this amounts to ask that the sheaf 
$\shHom[\stkc](\objc,\objc')$ is a locally free $\shEnd[\stkc](\objc')$-module of rank one. 

\begin{lemma}\label{lem:a+}
An $\shr$-stack $\stkc$ is equivalent to an $\shr$-algebra if and only if it
is non-empty and locally connected by isomorphisms.
\end{lemma}

\begin{proof}
One implication is clear. Suppose that $\stkc$ is non-empty and let  $\objc\in\stkc(X)$. Then the fully faithful functor $\astk{\shEnd[\stkc](\objc)}\to\stkc$ is an equivalence if and only if  $\stkc$ is locally connected
by isomorphisms.
\end{proof}

Let $\stkc$ be an $\shr$-stack. For $\shn\in \astk \shr$ and $\objc\in \stkc$, one defines 
$\shn \tens[\shr] \objc\in \stkc$ as the representative of $\shn \tens[\shr] \hom[\stkc](\cdot,\objc)\in  \stkMod(\stkc^\op)$. Then one has $\shr$-equivalences
\begin{align*}
\astk\shr \tens[\shr] \stkc  \equi \stkc, &  \qquad (\shn, \objc) \mapsto \shn \tens[\shr] \objc,\\
\stkc \equi \stkFun[\shr](\astk\shr,\stkc),  & \qquad \objc \mapsto  (\dummy) \tens[\shr] \objc.
\end{align*}

\begin{lemma}
\label{lem:ModR}
The definition \eqref{eq:Mod} of stack of $\stkc$-modules  does not
depend on the base ring $\shr$.
\end{lemma}

\begin{proof}
It follows from Lemma~\ref{lem:stktensadj} for $\shs = \Z_X$,
$\stkd = \astk\shr$
and $\stke = \stkMod(\Z_X)$ that 
\[
\stkFun[\shr](\stkc,\stkMod(\shr))
\equi 
\stkFun[\Z_X](\stkc,\stkMod(\Z_X)),
\]
where we use the equivalence
$\stkFun[\Z_X](\astk\shr,\stkMod(\Z_X)) \equi \stkMod(\shr).$
\end{proof}

\subsection{Compatibility}

Let $\sha$ and $\shb$ be two $\shr$-algebras, and $\varphi\colon Y \to X$ a continuous map 
(or a morphism of sites). There are an $\shr$-algebra isomorphism
\[
Z(\sha) \isoto Z(\astk \sha), \qquad
a \mapsto (\shn \to \shn ; n \mapsto an),
\]
and $\shr$-equivalences
\begin{align*}
(\astk\sha)^\op &\equi \astk{(\sha^\op)}, \qquad
\shn \mapsto \shHom[\sha^\op](\shn, \sha), \\
\stkMod(\sha) &\equi \stkMod(\astk\sha), \qquad
\shm \mapsto (\dummy) \tens[\sha] \shm, \\
\astk\sha \tens[\shr] \astk\shb &\equi  \astk{(\sha \tens[\shr] \shb)}, \qquad
(\shn , \shq) \mapsto \shn\tens[\shr]\shq\\
\opb \varphi \astk\sha &\equi \astk{(\opb \varphi \sha)}, \qquad
\shn  \mapsto \opb\varphi \shn.
\end{align*}

For $f,f'\colon \sha\to \shb$ two $\shr$-algebra morphisms, the sections on $U\subset X$ 
of the sheaf $\hom[{\stkFun[\shr](\astk \sha,\astk \shb)}]( \astk f , \astk {f'} )$ are given by
\begin{equation}
\label{eq:shbfunct}
\{ b\in \shb(U) \colon b f(a) = f'(a) b  \, \text{ for each } a\in
 \sha(V) \text{ and } V\subset U\},
\end{equation}
with composition of transformations given by the product in $\shb$. 

For $\shn$ a left $\shb$-module, with notation~\eqref{eq:2mod} we set
\begin{equation}
\label{eq:2modalg}
 {}_{f}\shn = {}_{\astk f}\shn
\end{equation}
the associated left $\sha$-module. By~\eqref{eq:shbfunct}, the morphism of $\sha$-modules associated 
to a transformation $b\colon \astk f \to \astk {f'}$ is given by
\begin{align*}
{}_{f}\shn \to {}_{f'}\shn; n\mapsto bn.
\end{align*}

\subsection{Algebroids}\label{sse:algebroid}

Recall from Lemma~\ref{lem:a+} that an $\shr$-stack is equivalent to
an $\shr$-algebra if and only if it is non-empty and is locally connected
by isomorphisms.

\begin{definition}\label{def:algebroid}
An $\shr$-algebroid is an $\shr$-stack which is locally non-empty and
locally connected by isomorphisms.
\end{definition}

In other words, an $\shr$-algebroid is an $\shr$-stack $\stka$ which is locally equivalent 
to an algebra. It is globally an algebra if and only if it has a global object\footnote{If the category 
$\stka(U)$ has a zero object for $U\subset X$, then $\stka|_U\equi[]\astk 0$, where $0$ denotes the ring with $1=0$. In particular, except for the case $\astk 0$, algebroids are not stacks of additive categories.}.
Note also that an $\shr$-stack is an $\shr$-algebroid if and only if its substack with the same 
objects and only invertible morphisms is a gerbe.

The stack $\stkMod(\stka)$ is an example of stack of twisted sheaves, i.e.~it is
a stack locally equivalent to a stack of modules over an algebra (see~\cite{KS06,DP04}).
A cocyclic description of algebroids and of their modules is recalled in 
Appendix~\ref{se:acocy} and~\ref{se:mcocy}.

Note that the existence of an $\shr$-functor
$\astk\shr\to\stka$ is equivalent to the existence of a global object
for $\stka$. In this case there is a forgetful functor
\[
\stkMod(\stka) \to \stkMod(\shr).
\]

\begin{lemma}
An $\shr$-stack $\stkc$ on $X$ is an algebroid if and only if $\pi_0(\stkc)\simeq\{\mathrm{pt}\}_X$, the singleton-valued constant sheaf.
\end{lemma}
It follows from~\eqref{eq:inv-pi_0} that inverse images of algebroids are algebroids.

Let $\stkc$ be an $\shr$-stack. Then for any $\shr$-algebroid $\stka$ one has
$$
\pi_0(\stka\tens[\shr]\stkc)\simeq\pi_0(\stkc).
$$
In particular, the tensor product of algebroids is an algebroid. 

\begin{definition}
\begin{itemize}
\item[(i)]
Let $\sha$ be an $\shr$-algebra. An $\shr$-twisted form of $\sha$ is
an $\shr$-algebroid locally $\shr$-equivalent to $\sha$.
\item[(ii)]
An invertible $\shr$-algebroid is an $\shr$-twisted form of $\shr$.
\end{itemize}
\end{definition}
Note that any $\shr$-functor between invertible $\shr$-algebroids is an equivalence,
since it is locally isomorphic to the identity functor of $\astk\shr$.

If $\stkc$ is an invertible $\shr$-algebroid, then $\shr\isoto Z(\stkc)$ and for any $\shr$-stack $\stkd$ 
there is an $\shr$-equivalence
$$
\stkc^\op\tens_\shr\stkd \equi \stkFun[\shr](\stkc,\stkd), \quad (\gamma,\delta)\mapsto \hom[\stkc](\gamma,\cdot) \tens[\shr] \delta.
$$
In particular, the set of $\shr$-equivalence classes of invertible $\shr$-algebroids 
is a group, with multiplication given by $\tens[\shr]$ and inverse given by $(\cdot)^\op$.

For an $\shr$-stack $\stkc$, denote by $\stkAut[\shr](\stkc)$ the full substack 
of $\stkAut(\stkc)$ whose objects are $\shr$-equivalences.
By Corollary~\ref{cor:h1aut}, the cohomology $H^1(X;\stkAut[\shr](\astk\sha))$ classifies
$\shr$-equivalence classes of $\shr$-twisted forms of $\sha$.
In terms of crossed modules (cf. Section~\ref{se:crossed-mod}), one has
\begin{equation}\label{eq:aut}
\stkAut[\shr](\astk\sha) \equi[] [(\sha^\times\to[\ad]
\shAut[\shr\text-\stack{Alg}_X](\sha),\ \delta)],
\end{equation}
where $\delta(f)(a) = f(a)$ and $\ad(a)(b) = a b a^{-1}$.
In particular,
$\stkAut[\shr](\astk\shr) \equi[] \shr^\times[1]$ and~\eqref{eq:HiG} implies

\begin{lemma}
\label{lem:tors}
The group of $\shr$-equivalence classes of
invertible $\shr$-alge\-broids is isomorphic to $H^2(X;\shr^\times)$.
\end{lemma}

\subsection{Inner forms}\label{subs:inn}

Let $\sha$ be a central $\shr$-algebra, i.e.~$Z(\sha)=\shr$.
(If $\sha$ is not central, the following discussion still holds by replacing $\shr$ with $Z(\sha)$.)

Denote by $\inn(\sha)$ the sheaf theoretical image of the group morphism
$\ad\colon \sha^\times\to\shAut[\shr\text-\stack{Alg}_X](\sha)$. Its sections are 
the inner automorphisms of $\sha$, i.e.~automorphisms locally of the form $\ad(a)$ 
for some $a\in\sha^\times$. Recall that an $\shr$-algebra $\shb$ is called an inner form of $\sha$ 
if there exists an open cover $\{U_i\}_{i\in I}$ of $X$ and ring isomorphisms
$f_i\colon \sha|_{U_i} \to \shb|_{U_i}$ such that $f_j^{-1}f_i \in \inn(\sha|_{U_{ij}})$. 

Examples of inner forms are given by Azumaya algebras 
and rings of twisted differential operators (see for example~\cite{DP04} for more details).

Let $\shb$ be an $\shr$-algebra. Denote by $\stke_{\sha,\shb}\subset \stkFun[\shr](\astk\sha,\astk\shb)$ the full substack of $\shr$-equivalences. Note that 
$\stke_{\sha,\shb}^\op\equi\stke_{\shb,\sha}$.

\begin{lemma}\label{le:AinnerB} 
$\shb$ is an inner form of $\sha$ if and only if  $\stke_{\sha,\shb}$ is an $\shr$-algebroid.
\end{lemma}

\begin{proof}
Since $\shr$-equivalences $\astk\sha\equito\astk\shb$ are locally given by $\shr$-algebra isomorphisms $\sha\isoto\shb$, it follows that $\stke_{\sha,\shb}$ is locally non-empty if and only if 
$\shb$ is locally isomorphic to $\sha$. 

Let $f, f '\colon\sha\to\shb$ be $\shr$-algebra 
isomorphisms. Setting $a=f^{-1}(b)$ in~\eqref{eq:shbfunct}, the invertible transformations from $\astk f$ to $\astk {f'}$ are given by
$$
\{ a\in \sha^\times \colon f^{-1}f' = \ad(a) \},
$$
hence $\stke_{\sha,\shb}$ is an $\shr$-algebroid if and only if $\shb$ is an inner form of $\sha$. 
\end{proof}
Since $\shEnd[\stke_{\sha,\shb}](\astk f)= \shr$, if $\shb$ is an inner form of $\sha$ it follows that 
$\stke_{\sha,\shb}$ is an invertible $\shr$-algebroid and $\stke_{\sha,\shb}\tens[\shr]\astk\sha \equi\astk\shb$. In particular, one gets an equivalence of stacks of 2-groups $\stkAut[\shr](\astk \sha)\equi[] \stkAut[\shr](\astk \shb)$.

Consider the non-abelian exact sequence
\[
H^1(X; \sha^\times) \to[b] H^1(X; \inn(\sha)) \to[c] H^2(X;\shr^\times)
\]
induced by the short exact sequence
\[
1\to \shr^\times \to \sha^\times \to[\ad] \inn(\sha) \to 1.
\]
For $\shb$ an inner form of $\sha$ and $\shp$ a locally free $\sha^\op$-module  of rank one, denote by $[\shb]$ and $[\shp]$ the associated cohomology classes
in $H^1(X;\inn(\sha))$ and $H^1(X; \sha^\times)$ respectively.
Then $b[\shp]=[\shEnd[\sha^\op](\shp)]$ and $c([\shb])=[\stke_{\sha,\shb}]$.

\begin{proposition}
\label{pro:AequiB}
The following conditions are equivalent.
\begin{itemize}
\item[(i)]
The stacks $\astk\sha$ and $\astk\shb$ are $\shr$-equivalent.
\item[(ii)]
There exists a locally free $\sha^\op$-module $\shp$ of rank one
such that  $\shb\simeq\shEnd[\sha^\op](\shp)$.
\item[(iii)]
$\shb$ is an inner form of $\sha$ and $c([\shb])=1$.
\end{itemize}
\end{proposition}

\begin{proof}
(i)$\Rightarrow$(ii)\footnote{The equivalence between (i) and (ii) can also be deduced
from Corollary~\ref{cor:AutInv}.}
Let $\functg\colon\astk\shb\to\astk\sha$ be an $\shr$-equivalence.
Recall that $\astk\shb \subset \stkMod(\shb^\op)$ is the substack of
locally  free modules of rank one.
Let $\objb$  be the canonical global object of $\astk\shb$, and set $\shp =
\functg(\objb)$. Then $\shb$ is isomorphic to $\shEnd[\sha^\op](\shp)$.

\smallskip\noindent
(ii)$\Rightarrow$(iii)
$\shb$ is clearly an inner form of $\sha$ and $\shp$ has a structure of $\sha^\op\tens[\shr]\shb$-module by the isomorphism $\shb\isoto\shEnd[\sha^\op](\shp)$. Then $(\cdot)\tens[\shb]\shp$ gives a global object of $\stke_{\shb,\sha}$ and $c([\shb])=[\stke_{\shb,\sha}^\op]=1$.

\smallskip\noindent
(i)$\Leftarrow$(iii)
By Lemma~\ref{le:AinnerB} follows that $c([\shb])=1$ if and only if $\stke_{\sha,\shb}$ has a global object.

\smallskip\noindent
\end{proof}

\section{Morita theory for linear stacks}\label{se:Mor}

Morita theory classically deals with modules over algebras.
It is extended to modules over linear categories in~\cite{Mit65,Mit85}
and to stacks of modules over sheaves of algebras
in~\cite[Chapter~19]{KS06} (see also \cite{DP04}). 
Here, we summarize these extensions by considering stacks of modules over linear
stacks, and in particular over algebroids.

\medskip

Let $X$ be a topological space (or a site), and $\shr$ a sheaf of
commutative rings on $X$. 

\subsection{Yoneda embedding}\label{sse:Y}
Recall that a category is called (co)complete if it admits small (co)limits. A
prestack $\stkc$ on $X$ is called (co)complete if the categories $\stkc(U)$
are (co)complete for each $U\subset X$, and the restriction functors commute
with (co)limits.

A prestack $\stkc$ on $X$ is called a proper stack (see~\cite{KS01,Pre07}) if
it is separated, cocomplete, and if for each inclusion of open subsets
$v\colon V \hookrightarrow U$, the restriction functors $\stkc(v) = (\cdot)|_V$ admits a
fully faithful left adjoint
\[
\eim v \colon \stkc(V) \to \stkc(U),
\]
called zero-extension, such that for a diagram of open inclusions
\[ 
\xymatrix@=1em{ V\cap W \ar[r]^-{v'} \ar[d]_{w'} & W \ar[d]^-w \\ V
\ar[r]^v & U, }
\] 
the natural transformation $\eim{v'} \circ \stkc(w') \to \stkc(w) \circ \eim
v$ is an isomorphism.

\begin{lemma}
\label{eq:vpropersheaf}
Let $\stkc$ be a proper stack.
For $\objc\in\stkc(V)$ and $\objc'\in\stkc(U)$
there is an isomorphism of $\shr|_U$-modules
\[
\oim v\hom[\stkc|_V](\objc,\objc'|_V) \simeq
\hom[\stkc|_U](\eim v\objc,\objc').
\] 
\end{lemma}
Recall that proper stacks are stacks.

\begin{lemma}\label{lem:modproper}
For any $\shr$-stack $\stkc$, the $\shr$-stack $\stkMod(\stkc)$ is proper and complete.
\end{lemma}

\begin{proof}
Recall first that $\stkMod(\shr)$ is complete and cocomplete. It is also
proper. In fact, for $v\colon V \hookrightarrow U$ an open inclusion, the
restriction functor of $\stkMod(\shr)$ coincides with the sheaf-theoretical
pull-back $\opb v$. This admits the direct image functor $\oim v$ as right
adjoint, and the zero-extension functor $\eim v$ as a left adjoint.

The statement follows, as $\stkMod(\stkc) =
\stkFun[\shr](\stkc,\stkMod(\shr))$ inherits the properties and structures of
$\stkMod(\shr)$. For example, for $v\colon V\hookrightarrow U$ an inclusion of
open subsets, the functor $\eim v \colon\catMod(\stkc|_V) \to
\catMod(\stkc|_U)$ is given by $(\eim v\shm)(\objc) = \eim u(\shm(\objc
|_{V\cap W}))$, where $\shm\colon \stkc|_V\to\stkMod(\shr|_V)$ is a
$\stkc|_V$-module, $W\subset U$ is an open subset, $\objc\in\stkc(W)$, and
$u\colon V\cap W \to U$ is the embedding.
\end{proof}

Let $\stkc$ be an $\shr$-stack.
The (linear) Yoneda embedding
is the full and faithful $\shr$-functor
\begin{equation}
\label{eq:stkyoneda}
\yoneda\stkc \colon \stkc^\op \to \stkMod(\stkc), \quad
\objc\mapsto \shHom[\stkc](\objc, \dummy)
\end{equation}
whose essential image are the functors $\stkc\to\stkMod(\shr)$ which are
representable. 
In analogy with the case $\stkc=\astk\sha$ for $\sha$ an $\shr$-algebra, 
a module $\shm\in\stkMod(\stkc)$ which is representable is called locally free of rank one.

As in the classical case, the full faithfulness of~\eqref{eq:stkyoneda} follows from

\begin{lemma}\label{lem:Yoneda}
For $\shm\in\catMod(\stkc)$ there is an isomorphism of $\stkc$-modules
$$
\shm(\dummy) \simeq \hom[\stkMod(\stkc)](\yoneda \stkc(\dummy), \shm).
$$
\end{lemma}

Let $\stkc$ be an $\shr$-stack. We use the notation
\begin{equation}
\label{eq:bimod}
\bimod\stkc \in \stkMod(\stkc^\op\tens_\shr\stkc)
\end{equation}
for the canonical object $\hom[\stkc](\cdot,\cdot)$.
This corresponds to the Yoneda embedding $\yoneda \stkc$
via the equivalence induced by Lemma~\ref{lem:stktensadj}
\[ 
\stkMod(\stkc^\op\tens_\shr\stkc) \equi \stkFun[\shr](\stkc^\op,\stkMod(\stkc)).
\] 
If $\stkc=\astk\sha$, the object in \eqref{eq:bimod} coincides with
$\sha$, considered as a bimodule over itself. 
If $\stkc$ is an invertible $\shr$-algebroid, then $\stkc^\op\tens_\shr\stkc\equi[\shr] \astk \shr$ 
and $\stkc$ is isomorphic to $\shr$ as a bimodule over itself.

\begin{lemma}\label{lem:ZMod}
For $\stkc$ an $\shr$-stack, there is a natural isomorphism of $\shr$-algebras
$Z(\stkc) \simeq Z(\stkMod(\stkc))$.
\end{lemma}

\begin{proof}
Recall that the $2$-functor $\stkMod(\dummy)$ defines a morphism of $\shr$-algebras 
$Z(\stkc)  \to Z(\stkMod(\stkc))$, sending a transformation $\transfd$ of $\id_{\stkc}$ to the transformation of
$\id_{\stkMod(\stkc)}$ whose morphism associated to $\shn\in\stkMod(\stkc)$ is $\id_\shn \mathop\bullet \transfd$.
(Here $\bullet$ denotes the horizontal composition of transformations). 
Hence we get a commutative diagram
\[
\xymatrix@R=2em@C=5em{
Z(\stkc) \ar@{=}[d] \ar[r] & Z(\stkMod(\stkc)) \ar[d]^{\dummy \mathop\bullet \id_{\yoneda \stkc}}  \\
Z(\stkc^\op) \ar[r]^{\id_{\yoneda \stkc} \mathop\bullet \dummy \qquad} & \shendo[{\stkFun[\shr](\stkc^\op,\stkMod(\stkc))}](\yoneda \stkc)
}
\]
where $\id_{\yoneda \stkc} \mathop\bullet \dummy$ is an isomorphism, since~\eqref{eq:stkyoneda} is fully faithful, and 
$\dummy \mathop\bullet \id_{\yoneda \stkc}$ is injective, since by Lemma~\ref{lem:Yoneda} a transformation $\transfc$ of $\id_{\stkMod(\stkc)}$ is completely determined by $\transfc\mathop\bullet \id_{\yoneda \stkc}$.
\end{proof}

\subsection{Operations}

Let $\stkc$ be an $\shr$-stack.
As for modules over a sheaf of algebras, there is a natural $\shr$-functor
\begin{equation}
\label{eq:tensC}
\tens_\stkc \colon \stkMod(\stkc^\op) \tens_\shr
\stkMod(\stkc) \to \stkMod(\shr).
\end{equation}
This is discussed in \cite{Mit85} when $X$ is reduced to a point.
In order to explain how this extends to stacks, we need some preparation.

Denote by $\stkc/X$ the Grothendieck construction associated with $\stkc$. Recall that
objects of $\stkc/X$ are pairs $(u,\objc)$ with $u\colon U\to X$
an open inclusion, and $\objc\in\stkc(U)$. Morphisms $c\colon (u,\objc)
\to (u',\objc')$ are defined only if $U'\subset U$, and in that case are given
by morphisms $\objc\vert_{U'}\to \objc'$ in $\stkc(U')$. For $c'\colon
(u',\objc') \to (u'',\objc'')$ another morphism, the composition\footnote{Here we
denote for short by $c|_{U''}$ the composite $\objc|_{U''} \isofrom
\objc|_{U'}|_{U''} \to[c|_{U''}] \objc'|_{U''}$.} 
is given by $c' \circ c|_{U''}$.

\begin{notation}
Let $\catc$ be a category. We denote by $\Moro \catc$ the category whose objects are morphisms $c\colon \objc \to \objc'$ in $\catc$ and whose morphisms $c\to d$ are pairs $(e,e')$ of morphisms in $\catc$ such that $c=e'\circ d\circ e$. This is visualised by the diagram
\[
\xymatrix{
\objc \ar[r]^c \ar[d]_e & \objc'  \\
\objd \ar[r]^d & \objd' \ar[u]_{e'}.
}
\]
\end{notation}

The following lemma is a stack-theoretical analogue of \cite[Lemma 2.1.15]{KS06}.

\begin{lemma}
\label{le:stkfunlim}
Let $\shm,\shn\in\catMod(\stkc)$. Then
\begin{itemize}
\item[(i)] 
the assignment
\[
\bigl( (u,\objc)
\to[c] (u',\objc') \bigr) \mapsto
\hom[\shr](\oim u\shm(\objc), \oim u' \shn(\objc'))
\]
defines a functor $(\Moro \stkc/X)^\op \to \catMod(\shr)$;
\item[(ii)] 
there is an isomorphism in $\catMod(\shr)$
\begin{equation}
\label{eq:MhomN}
\hom[\stkMod(\stkc)](\shm, \shn) 
\simeq \plim[{\bigl( (u,\objc)
\to[c] (u',\objc') \bigr)\in \Moro \stkc/X}] \hom[\shr](\oim u\shm(\objc), \oim u' \shn(\objc')).
\end{equation}
\end{itemize}
\end{lemma}

\begin{proof}
(i) Let us check that a morphism $(e,e')$ in $\Moro \stkc/X$, visualised by the diagrams
\[
\xymatrix{
U & U' \ar[l] \ar[d]  \\
V \ar[u] & V' \ar[l],
}
\qquad
\xymatrix{
(u,\objc) \ar[r]^c \ar[d]_e & (u',\objc')  \\
(v,\objd) \ar[r]^d & (v',\objd') \ar[u]_{e'},
}
\]
induces a morphism in $\catMod(\shr)$
\begin{equation}
\label{eq:moree'}
\hom[\shr](\oim v\shm(\objd), \oim v' \shn(\objd')) \to
\hom[\shr](\oim u\shm(\objc), \oim u' \shn(\objc')).
\end{equation}
Consider the inclusions of open subsets
\[
\xymatrix@R=1em@C=.5em{
X  && U \ar[ll]_u\\
&& V \ar[ull]^v \ar[u]_i.
}
\]
The morphism $e\colon \objc|_V \to \objd$ induces a morphism
\[
\oim u\shm(\objc) \to \oim u\oim i\opb i\shm(\objc) \simeq \oim v\shm(\objc|_V) \to \oim v\shm(\objd).
\]
Similarly, $e'\colon \objd'|_{U'} \to \objc'$ induces a morphism $\oim v'\shn(\objd') \to \oim u'\shn(\objc')$.
Then \eqref{eq:moree'} is obtained by left and right composition with these morphisms.

(ii) It is enough to prove that the natural morphism in $\catMod(\shr(X))$
\begin{align*}
\Hom[\stkMod(\stkc)](\shm, \shn) 
&\to \plim[c\in \Moro \stkc/X] \Hom[\shr|_{U'}](\shm(\objc|_{U'}), \shn(\objc')) \\
&\simeq \plim[c\in \Moro \stkc/X] \Hom[\shr](\oim u\shm(\objc), \oim u' \shn(\objc')) 
\end{align*}
is an isomorphism.
The proof of this fact follows the same arguments as those in the proof of \cite[Lemma 2.1.15]{KS06}.
\end{proof}

For $\shm\in\catMod(\stkc)$ and $\shp\in\catMod(\stkc^\op)$,
similarly as above,
the assignment
\[
\bigl( (u,\objc)
\to[c] (u',\objc') \bigr) \mapsto
\eim u'\shp(\objc') \tens[\shr] \oim u \shm(\objc)
\]
defines a functor $\Moro \stkc/X \to \catMod(\shr)$ and we set
\begin{equation}
\label{eq:PtensM}
\shp \tens[\stkc] \shm =
\ilim[{\bigl( (u,\objc)
\to[c] (u',\objc') \bigr)\in \Moro \stkc/X}] \eim u'\shp(\objc') \tens[\shr] \oim u \shm(\objc).
\end{equation}

For $\sha$ an $\shr$-algebra and $\stkc=\astk \sha$, this is the usual tensor product of 
$\shm\in\catMod(\sha)$ and $\shp\in\catMod(\sha^\op)$. For example, when $X$ is reduced to a point, this amounts to present $\shp\tens_\sha \shn$ as the
quotient of $\DSum_{a\in\sha}\shp\tens_\shr \shn$ by the subgroup generated by the elements $(p\tens m)_a - (p\tens am)_1$
and $(p\tens m)_a - (pa\tens m)_1$ for $a\in\sha$, $p\in\shp$ and $m\in\shm$. (The
indices denote the direct summand to which the elements belong. These generators correspond to the morphisms
$(a,1),(1,a)\colon a\to 1$ in $\Moro \astk \sha$.)

\begin{lemma} 
$\shp \tens[\stkc] \shm$ is a representative of the functor 
\[
\hom[\stkMod(\stkc)](\shm,\hom[\shr](\shp,\cdot)) \colon \stkMod(\shr) \to \stkMod(\shr),
\]
where, for any $\shl\in\stkMod(\shr)$, we denote by $\hom[\shr](\shp,\shl)$ the $\stkc$-module
$\hom[\shr](\shp(\dummy), \shl)$.
\end{lemma}

\begin{proof}
By Lemma~\ref{le:stkfunlim}, for any $\shl\in\stkMod(\shr)$ there are isomorphisms
\begin{align*}
\hom[\stkMod(\stkc)]&(\shm,\hom[\shr](\shp,\shl))\simeq \\
&\simeq \plim[c\in \Moro \stkc/X] \hom[\shr](\oim u\shm(\objc), \oim u' \hom[\shr|_{U'}](\shp(\objc'),\shl|_{U'})) \\
&\simeq \plim[c\in \Moro \stkc/X] \hom[\shr](\oim u\shm(\objc),  \hom[\shr](\eim u'\shp(\objc'),\shl)) \\
&\simeq \plim[c\in \Moro \stkc/X] \hom[\shr](\eim u'\shp(\objc') \tens[\shr] \oim u\shm(\objc),\shl) \\
&\simeq \hom[\shr]\left(\ilim[c\in \Moro \stkc/X] \eim u'\shp(\objc') \tens[\shr] \oim u\shm(\objc),\shl\right)\\
&= \hom[\shr](\shp \tens[\stkc] \shm,\shl).
\end{align*}
\end{proof}

Uniqueness of representatives imply the functoriality of the assignment $(\shp,\shm)\mapsto \shp\tens[\stkc]\shm$. We have
thus defined the functor \eqref{eq:tensC}.

As for modules over a sheaf of algebras, let us use the short hand notation
\begin{equation}
\label{eq:homC}
\hom[\stkc](\cdot,\cdot) \colon \stkMod(\stkc)^\op \tens_\shr \stkMod(\stkc) \to \stkMod(\shr).
\end{equation}
for the internal hom-functor $\hom[\stkMod(\stkc)](\cdot,\cdot)$.

\begin{notation}
\label{not:tens}
We use the same notation
\begin{align*}
\hom[\stkc] &\colon \stkMod(\stkc\tens_\shr\stkd^\op)^\op \tens_\shr
\stkMod(\stkc\tens_\shr\stke) \to \stkMod(\stkd\tens_\shr\stke), \\
\tens_\stkc & \colon \stkMod(\stkc^\op\tens_\shr\stkd) \tens_\shr
\stkMod(\stkc\tens_\shr\stke) \to \stkMod(\stkd\tens_\shr\stke)
\end{align*}
for the $\shr$-functors obtained by ``picking up operators'' (in the sense of \cite[page 15]{Mit72}) 
from the $\shr$-functors \eqref{eq:homC} and \eqref{eq:tensC}.
\end{notation}

These functors satisfy the relations \eqref{eq:MhomN} and \eqref{eq:PtensM}. For example, let
$\shp\in\stkMod(\stkc^\op\tens_\shr\stkd)$ and $\shm\in\stkMod(\stkc\tens_\shr\stke)$. Consider them as functors $\shp\colon
\stkc^\op \to \stkMod(\stkd)$ and $\shm\colon \stkc \to \stkMod(\stke)$. Then $\shp\tens[\stkc]\shm$ satisfies
\eqref{eq:PtensM}, where the operations $\eim{u'}$ and $\eim u$ are those of the proper stacks $\stkMod(\stkd)$ and
$\stkMod(\stke)$, respectively, and the tensor product is the natural functor 
\[ 
\tens_\shr \colon \stkMod(\stkd) \tens_\shr
\stkMod(\stke) \to \stkMod(\stkd\tens_\shr\stke). 
\]

The standard formulas concerning the usual hom-functor and tensor product hold.
For example,

\begin{lemma}
\label{le:tenshomadj}
For $\shm\in\stkMod(\stkc^\op\tens_\shr\stkd)$,
$\shn\in\stkMod(\stkc\tens_\shr\stke)$, and
$\shp\in\stkMod(\stkd\tens_\shr\stkf)$, there is an isomorphism
in $\stkMod(\stke^\op\tens_\shr\stkf)$
\[ 
\hom[\stkd](\shm\tens_\stkc \shn, \shp) \simeq 
\hom[\stkc](\shn, \hom[\stkd](\shm,\shp)).
\]
\end{lemma}

Recall the notation $\bimod\stkc \in \stkMod(\stkc^\op\tens_\shr\stkc)$ in \eqref{eq:bimod}.
By Lemma~\ref{lem:Yoneda} the functors $\hom[\stkc] (\bimod\stkc,\dummy)$, and hence also
$\bimod\stkc\tens_\stkc (\dummy)$, are isomorphic to the identity.

Using \eqref{eq:PtensM} with $\shp=\stkc$, we get

\begin{lemma}
\label{pr:YMprstk}
For $\shm\in\catMod(\stkc)$ there is an isomorphism in $\catMod(\stkc)$
\[ 
\shm \simeq
\ilim[{\bigl( (u,\objc)
\to[c] (u',\objc') \bigr)\in \Moro \stkc/X}] \eim u'\yoneda \stkc(\objc') \tens[\shr] \oim u \shm(\objc).
\]
\end{lemma}

\subsection{Morita equivalence}\label{sse:Morita}

Let us discuss how classical Morita theory extends to linear stacks.

Let $\stkc$ and $\stkd$ be $\shr$-stacks.
Denote by $\stkFun[\shr]^r(\stkMod(\stkc),\stkMod(\stkd))$ the stack of 
$\shr$-functors admitting a right adjoint. 
The Eilenberg-Watts theorem for $\shr$-stacks holds:

\begin{theorem}
\label{th:premorita}
\begin{itemize}
\item[(i)] 
The $\shr$-functor
\[ 
\stkMod(\stkc^\op\tens_\shr\stkd) \to
\stkFun[\shr]^r(\stkMod(\stkc),\stkMod(\stkd)), \quad \shp \mapsto
\shp\tens_\stkc(\dummy),
\] 
is an equivalence.
\item[(ii)] 
For $\shp\in\stkMod(\stkc^\op\tens_\shr\stkd)$ and
$\shq\in\stkMod(\stkd^\op\tens_\shr\stke)$ one has an isomorphism in $\stkFun[\shr]^r(\stkMod(\stkc),\stkMod(\stke))$
\[
(\shq\tens[\stkd]\shp) \tens[\stkc] (\cdot) \simeq (\shq \tens[\stkd]
(\cdot)) \circ (\shp \tens[\stkc] (\cdot)).
\]
\end{itemize}
\end{theorem}

Given an $\shr$-functor $\functh\colon \stkMod(\stkc) \to \stkMod(\stkd)$, we will use the same 
notation $\functh$ for the induced $\shr$-functor, obtained by ``picking up operators'',
$$
\stkMod(\stkc^\op\tens[\shr]\stkc) \to \stkMod(\stkc^\op\tens[\shr]\stkd).
$$

\begin{proof}
(i) Let us show that $\functh\mapsto \functh(\stkc)$ is a quasi-inverse to the functor in the statement.

Since $\shp \simeq \shp\tens_\stkc\stkc$, we only have to prove that $\functh \isoto \functh(\stkc)\tens_\stkc(\dummy)$.
Let $\shm\in\stkMod(\stkc)$. Since $\functh$ has a right adjoint, it commutes with direct limits, proper direct images, and tensor products with $\shr$-modules.
Hence we have
\begin{align*}
\functh(\shm)
&\simeq \functh\left(\ilim[{c \in \Moro \stkc/X}] \eim u'\yoneda \stkc(\objc') \tens[\shr] \oim u \shm(\objc)\right) \\
&\simeq \ilim[{c \in \Moro \stkc/X}] \eim u'\functh(\yoneda \stkc(\objc')) \tens[\shr] \oim u \shm(\objc) \\
&\simeq \ilim[{c \in \Moro \stkc/X}] \eim u'\functh(\stkc)(\objc',\dummy) \tens[\shr] \oim u \shm(\objc) \\
& = \functh(\stkc) \tens[\stkc] \shm,
\end{align*}
where the first isomorphism follows from Lemma~\ref{pr:YMprstk} and we use the fact that $\functh \circ \yoneda \stkc$ identifies to $\functh(\stkc)$ via the equivalence induced by Lemma~\ref{lem:stktensadj}.

(ii) By (i), the statement amounts to the isomorphism in $\stkMod(\stkc^\op\tens_\shr\stke)$
\[
(\shq\tens[\stkd]\shp) \tens[\stkc] \stkc \simeq \shq \tens[\stkd] (\shp \tens[\stkc] \stkc).
\]
\end{proof}

\begin{remark}
Denoting by $\stkFun[\shr]^l(\stkMod(\stkc),\stkMod(\stkd))$ the stack of
$\shr$-functors admitting a left adjoint, one similarly gets an
$\shr$-equivalence
\[ 
\stkMod(\stkc^\op\tens_\shr\stkd) \to
\stkFun[\shr]^l(\stkMod(\stkd),\stkMod(\stkc))^\op, \quad \shp\mapsto
\hom[\stkd](\shp,\dummy),
\] 
and the corresponding commutative diagram as in
Theorem~\ref{th:premorita}~(ii). These constructions are interchanged by
the $\shr$-equivalence
\begin{equation*}
\stkFun[\shr]^r(\stkMod(\stkc),\stkMod(\stkd)) \equi
\stkFun[\shr]^l(\stkMod(\stkd),\stkMod(\stkc))^\op
\end{equation*}
sending a functor to its adjoint.
\end{remark}

\begin{definition}
\label{def:inv}
\begin{itemize}
\item[(i)] 
One says that $\shq\in\stkMod(\stkd^\op\tens_\shr\stkc)$ is an inverse of
$\shp\in\stkMod(\stkc^\op\tens_\shr\stkd)$ if there are isomorphisms of
$\stkc\tens_\shr\stkc^\op$- and $\stkd\tens_\shr\stkd^\op$-modules, respectively, 
\[
\shq\tens_\stkd \shp \simeq \bimod\stkc, \quad 
\shp\tens_\stkc \shq \simeq \bimod\stkd.
\]
\item[(ii)] 
An object $\shp\in\stkMod(\stkc^\op\tens_\shr\stkd)$ is called invertible if
it has an inverse.
\end{itemize}
\end{definition}

One proves (see e.g.~\cite[\S19.5]{KS06}) that $\shp$ is invertible if and
only if one of the following equivalent conditions is satisfied
\begin{itemize}
\item [(i)] 
$\hom[\stkd](\shp,\bimod\stkd)$ is an inverse of $\shp$;
\item [(ii)] 
the functor $\shp\tens[\stkc](\dummy) \colon \stkMod(\stkc) \to
\stkMod(\stkd)$ is an $\shr$-equivalence.
\item [(iii)] 
the functor $\hom[\stkc^\op](\shp,\dummy) \colon \stkMod(\stkc^\op) \to
\stkMod(\stkd^\op)$ is an $\shr$-equiva\-lence.
\end{itemize}

\begin{notation}
For any $\shr$-functor $\functf\colon \stkc \to \stkc'$, we denote by $\stkEndo[\stkc'](\functf)$ the 
$\shr$-stack associated to the separated prestack whose objects on $U\subset X$ are those 
of $\stkc(U)$ and $\Hom(\gamma,\gamma')=\Hom[\stkc'(U)](\functf(\gamma),\functf(\gamma'))$.
\end{notation}
Note that, if $\functf$ is fully faithful, then the natural $\shr$-functor $\stkc \to \stkEndo[\stkc'](\functf)$ induced by $\functf$ is an equivalence. In particular, identifying $\stkc\in\stkMod(\stkc^\op\tens_\shr\stkc)$ with the Yoneda embedding $\yoneda \stkc\colon \stkc^\op \to \stkMod(\stkc)$, one has $\stkc^\op\equi\stkEndo[ \stkMod(\stkc)](\stkc)$.
Moreover, considering $\shp\in\stkMod(\stkc^\op\tens_\shr\stkd)$ as a functor $\stkc^\op \to \stkMod(\stkd)$, the condition of $\shp$ being invertible is further equivalent to
\begin{itemize}
\item [(iv)] \label{pgv}
$\shp$ is a faithfully flat\footnote{$\shp$ is a faithfully flat $\stkd$-module if the functor $(\dummy)\tens[\stkd]\shp$ is faithful and exact.} $\stkd$-module locally of finite presentation\footnote{$\shp$ is a $\stkd$-module of finite presentation if the functor $\hom[\stkd](\shp,\dummy)$ commutes with small filtrant colimits.}
and $\stkc^\op\equi \stkEndo[\stkMod(\stkd)](\shp)$;
\item [(v)] 
$\shp$ is $\stkd$-progenerator\footnote{$\shp$ is $\stkd$-progenerator if the functor $\hom[\stkd](\shp,\dummy)$ is faithful and exact.} locally of finite type and $\stkc^\op\equi \stkEndo[\stkMod(\stkd)](\shp)$.
\end{itemize}
By reversing the role of $\stkc^\op$ and $\stkd$, one gets dual equivalent conditions.

\begin{theorem}[Morita]\label{th:morita}
An $\shr$-functor $\functh\colon \stkMod(\stkc) \to \stkMod(\stkd)$ is an
equivalence if and only if 
$\shp = \functh(\stkc)$ is an invertible $(\stkc^\op\tens_\shr\stkd)$-module. Moreover, one has
$\functh \simeq \shp\tens_\stkc(\dummy)$.
\end{theorem}

\begin{definition} 
Two $\shr$-stacks $\stkc$ and $\stkd$ are Morita $\shr$-equivalent if their
stacks of modules $\stkMod(\stkc)$ and $\stkMod(\stkd)$ are $\shr$-equivalent.
\end{definition}
Hence $\stkc$ and $\stkd$ are Morita $\shr$-equivalent if and only if there exists an invertible 
$(\stkc^\op\tens_\shr\stkd)$-module.

Let us say that $\shp\in\stkMod(\stkc^\op\tens[\shr]\stkd)\equi \stkFun[\shr](\stkd,\stkMod(\stkc^\op))$ 
is locally free of rank one over $\stkc^\op$ if for any $\objd\in\stkd$ the $\stkc^\op$-module
$\shp(\objd)$ is locally free of rank one, that is to say, the functor $\shp(\objd) \colon
\stkc^\op \to\stkMod(\shr)$ is representable.

Recall from~\eqref{eq:2mod} that 
${}_\functf(\dummy)\colon \stkMod(\stkc)\to\stkMod(\stkd)$ denotes
the functor associated to an $\shr$-functor $\functf\colon \stkd\to\stkc$.

\begin{proposition}\label{pro:AutInv}
The $\shr$-functor
\begin{equation}\label{eq:FctMod}
\stkFun[\shr](\stkd,\stkc) \to 
\stkMod(\stkc^\op\tens[\shr]\stkd),
\quad \functf\mapsto {}_\functf\bimod\stkc
\end{equation}
is fully faithful and induces an equivalence with the full substack of
locally free modules of rank one over $\stkc^\op$.
\end{proposition}

\begin{proof}
(i) The functor in the statement equals $\yoneda {\stkc^\op}\circ \cdot$.
This is fully faithful, since $\yoneda {\stkc^\op}$ is fully faithful.

(ii) Assume that $\shp\in\stkMod(\stkc^\op\tens[\shr]\stkd)$ is a locally
free module of rank one over $\stkc^\op$. Then $\shp \simeq
{}_\functf\bimod\stkc$, where $\functf\colon \stkd \to\stkc$ is the
functor associating to $\objd\in\stkd$ the representative of $\shp(\objd)$.
\end{proof}

\begin{corollary}\label{cor:AutInv}
Two $\shr$-stacks $\stkc$ and $\stkd$ are $\shr$-equivalent if and
only if there exists $\shp \in \stkMod(\stkc^\op\tens[\shr]\stkd)$
which is invertible and locally free of rank one over $\stkc^\op$.
\end{corollary}

In particular, two algebroids $\stka$ and $\stkb$ are $\shr$-equivalent if and only if there exists an
invertible $(\stka^\op\tens[\shr]\stkb)$-module $\shp$ which is locally free of
rank one over $\stka^\op$. These conditions on $\shp$ are equivalent
to the condition that $\shp$ is bi-invertible in the sense of
\cite[Corollary  2.1.10]{KS12}.

\begin{remark}\label{rem:LocFreeBim}
If $\stkc\equi\astk\sha$ and $\stkd\equi\astk\shb$, the functor  $\astk\shb \to\astk\sha$ 
associated to an $\sha^\op\tens[\shr]\shb$-module $\shp$ locally free of rank one 
over $\sha^\op$ is $\functf = (\cdot)\tens[\shb]\shp$. 
Note that any local isomorphism $h\colon \sha \isoto \shp$ of right $\sha$-modules defines
a local $\shr$-algebra morphism (isomorphism if $\shp$ invertible)
\begin{equation}\label{eq:LocFreeBim}
f\colon \shb\to \shEnd[\sha^\op](\shp) \to[\ad(\opb h)] \shEnd[\sha^\op](\sha)\simeq \sha,
\end{equation}
(the first arrow is induced by the $\shb$-module structure of $\shp$), for which
$h\colon {}_f\sha \isoto \shp$ is an isomorphism of $\sha^\op\tens[\shr]\shb$-modules and 
$\functf \simeq \astk f$. If $h$ is given by $a \mapsto ua$ for a local generator $u$ of the right $\sha$-modules $\shp$, then $f(b) = a$ for $a$ such that $ua=bu$.
\end{remark}

\subsection{Picard good stacks}

We will use the notation
\[
\stkc^e = \stkc^\op\tens_\shr\stkc.
\]

\begin{definition}
\label{def:Picgood}
An $\shr$-stack $\stkc$ is Picard good if all invertible $\stkc^e$-modules are locally free of rank one over $\stkc^\op$ (or, equivalently, over $\stkc$). An $\shr$-algebra $\sha$ is Picard good if it is so as an $\shr$-stack.
\end{definition}

For $\sha= \shr$, one recovers Definition 4.2 given in~\cite{DP04}.

Since the condition of being Picard good is local, an algebroid is Picard good if and only if so are the algebras that locally represent it.

Recall from (v) in Section~\ref{sse:Morita} that invertible bimodules are projective as right (or left) modules. It follows that examples of Picard good rings are  projective-free rings, and in particular local rings. Note however that Picard good does not imply projective-free (see Remark~\ref{rem:tf}).

Denote by $\stkInv(\stkc^e)$ the substack of $\stkMod(\stkc^e)$ whose objects
are invertible $\stkc^e$-modules and whose morphisms are only those morphisms which
are invertible. Then $\tens[\stkc]$ induces on $\stkInv(\stkc^e)$ a natural structure of stack 
of $2$-groups, and~\eqref{eq:FctMod} gives a fully faithful functor of
stacks of $2$-groups
\begin{equation}\label{eq:AutInv}
\stkAut[\shr](\stkc)_\op \hookrightarrow \stkInv(\stkc^e),
\quad \functf\mapsto {}_\functf\stkc.
\end{equation}
Here, for $\stkg$ a stack of $2$-groups, $\stkg_\op$ denotes the stack of
$2$-groups with the same groupoid structure as $\stkg$ and with reversed monoidal structure.

Set 
$$\out[\shr](\stkc) = \pi_0(\stkAut[\shr](\stkc)), \qquad \pic[\shr](\stkc) = \pi_0(\stkInv(\stkc^e)).$$
Then~\eqref{eq:AutInv} induces an injective homomorphism of groups
\begin{equation}\label{eq:InnPic}
\out[\shr](\stkc)^\op \hookrightarrow \pic[\shr](\stkc).
\end{equation}
Note that, from~\eqref{eq:aut} follows that $\out(\astk\sha)=\aut[\shr](\sha)/\inn(\sha)$, the sheaf of outer automorphisms of $\sha$.

\begin{proposition}\label{pro:Out}
Let $\stkc$ be an $\shr$-stack. Then the following are equivalent
\begin{itemize}
\item[(i)] $\stkc$ is Picard good;
\item[(ii)] \eqref{eq:AutInv} is an equivalence;
\item[(iii)] \eqref{eq:InnPic} is an isomorphism.
\end{itemize}
\end{proposition}

\begin{proof}
The equivalence between (i) and (ii) follows from Proposition~\ref{pro:AutInv}, whereas that 
between (ii) and (iii) follows from the fact that a functor of stacks $\stkc \to \stkc'$ is essentially 
surjective if and only if the induced morphism of sheaves $\pi_0(\stkc) \to \pi_0(\stkc')$ is surjective.
\end{proof}

\begin{proposition}\label{pro:picgood}
Let $\stkc$ be a Picard good $\shr$-stack.
\begin{itemize}
\item[(i)] 
Let $\stkd$ be an $\shr$-stack locally equivalent to  
$\stkc$. Then $\stkc$ and
$\stkd$ are Morita $\shr$-equivalent if and only if they are
$\shr$-equivalent.
\item[(ii)] 
Let $\stkm$ be an $\shr$-stack locally $\shr$-equivalent to $\stkMod(\stkc).$
Then $\stkm\equi\stkMod(\stkd)$ for an  $\shr$-stack  $\stkd$ locally
$\shr$-equivalent to $\stkc$.
\end{itemize}
\end{proposition}

\begin{proof}
(i) By Theorem~\ref{th:morita}, there is an
equivalence of stacks of $2$-groups
$$
\stkInv(\stkc^e) \equito \stkAut[\shr](\stkMod(\stkc)), \quad
\shp\mapsto \shp\tens[\stkc](\cdot).
$$
We thus have a (quasi-)commutative diagram 
\[
\xymatrix@C=0em@R=2em{
\stkInv(\stkc^e) \ar[rr]^-{\equi[]} &&
\stkAut[\shr](\stkMod(\stkc)) \\
& \stkAut[\shr](\stkc)_\op ,
\ar@{_(->}[ul] \ar@{^(->}_{\stack m}[ur],}
\]
where $\stack m$ is induced by the functor $\stkMod(\dummy)$. 
It follows from Proposition~\ref{pro:Out}
that $\stkc$ is Picard good if and only if $\stack m$ is an equivalence.

Let $\stkEquiv[\shr](\cdot,\cdot)$ denote the stack of $\shr$-equivalences, with invertible transformations as morphisms. Consider the functor
$$
\stkEquiv[\shr](\stkc,\stkd) \to 
\stkEquiv[\shr](\stkMod(\stkd),\stkMod(\stkc))
$$
induced by the 2-functor $\stkMod(\dummy)$. Since $\stkd$ is locally equivalent to $\stkc$, 
this locally reduces to the functor $\stack m$ above. It is thus locally, hence globally, an equivalence.

\smallskip\noindent (ii)
Let $\stke \subset \stkm$ be the full substack of objects $\gamma$ with the property that for any local $\shr$-equivalence $\functh\colon \stkm \equito \stkMod(\stkc)$, the $\stkc$-module 
$\functh(\gamma)$ is locally free of rank one.
Since $\stkc$ is Picard good, the $\shr$-stack $\stke$ is locally non-empty and locally $\shr$-equivalent to $\stkc^\op$.
Set $\stkd=\stke^\op$. Then the $\shr$-functor
$$\stkm \to \stkMod(\stkd), \quad \delta \mapsto \hom[\stkm](\cdot,\delta)$$
is locally, hence globally, an equivalence.
\end{proof}

If $\stkc$ is an invertible $\shr$-algebroid, then it is Picard good if and only if $\shr$ is, and one has equivalences of stacks of 2-groups
\begin{equation}\label{eq:TorsInv}
\shr^\times[1] \equito \stkInv(\shr)\equi[] \stkInv(\stkc^e),
\quad \shp \mapsto \shr \times_{\shr^\times} \shp.
\end{equation}
(Recall that $\shr^\times[1]$ denotes the stack of $\shr^\times$-torsors.)
Moreover, in this situation the stack $\stkd$ in $(ii)$ above is $\shr$-equivalent via 
$\gamma \mapsto \hom[\stkm](\gamma,\cdot)$  to the full substack of $\stkFun[\shr](\stkm,\stkMod(\shr))$ whose objects are equivalences.

Examples of stacks as in Proposition \ref{pro:picgood} (ii) arise from deformations of categories of modules as discussed in~\cite{LVdB06}. 
In particular, Proposition \ref{pro:picgood} applies when $\stkc$ is (equivalent to) the structure sheaf of a ringed space. We thus recover results of \cite{Low08}.

\section{Microdifferential operators}\label{se:mic}

We collect here some results from the theory of microdifferential operators of
\cite{S-K-K} (see also \cite{Sch85,Kas86,Kas03}). The statements about the automorphisms of the sheaf of microdifferential operators are well known. Since we lack a reference for the proofs, we give them here.

\subsection{Microdifferential operators}

Let $M$ be an $n$-dimensional complex manifold, $T^*M$ its cotangent bundle 
and $\dTM \subset T^*M$ the open subset obtained by removing the zero-section.

Denote by $\she_{\dTM}$ the sheaf of microdifferential operators on
$\dTM$. Recall that $\she_{\dTM}$ is a sheaf of central $\C$-algebras
endowed with a $\Z$-filtration by the order of the operators, and one has
$$\gr\she_{\dTM}\simeq \DSum_{m\in\Z}\sho_{\dTM}(m),$$ 
where $\sho_{\dTM}(m)$ is the subsheaf of $\sho_{\dTM}$ of holomorphic functions 
homogeneous of degree $m$.

For $\lambda\in\C$, denote by $\she_\dTM(\lambda)$ the sheaf of microdifferential 
operators of order at most $\lambda$. In a local coordinate system $(x)$ on $M$, with associated symplectic
coordinates $(x;\xi)$ on $\dTM$, a section $P\in\sect(V;\she_\dTM(\lambda))$
is determined by its total symbol, which is a formal series
\[
\tot(P)=\sum_{j=0}^{+\infty} p_{\lambda - j}(x,\xi)
\]
with $p_{\lambda - j}\in\sect(V;\sho_\dTM)$ homogeneous of degree ${\lambda - j}$,
satisfying suitable growth conditions in $j$. If $Q$
is a section of $\she_\dTM(\mu)$, then $PQ\in \she_\dTM(\lambda + \mu)$ has
total symbol given by the Leibniz formula
\[
\tot(P Q)=
\sum_{\alpha\in \N^n} \frac{1}{\alpha !} 
\partial^{\alpha}_\xi\tot(P)
\partial^{\alpha}_x\tot(Q).
\]

Set
\[
\she_\dTM^{[\lambda]} = \Union_{n\in\Z}\she_\dTM(\lambda+n),
\]
where $[\lambda]$ is the class of $\lambda$ in $\C/\Z$, and denote by
\[
\sigma_\lambda\colon\she_\dTM(\lambda)\to\sho_\dTM(\lambda) \quad\text{and}\quad
\sigma\colon\she_\dTM^{[\lambda]}\to\sho_\dTM
\]
the symbol of order $\lambda$ and the principal symbol, respectively, where
$\sigma(P)=\sigma_\mu(P)$ for
$P\in\she_\dTM(\mu)\setminus\she_\dTM(\mu-1)$.
Note that $\she_\dTM^{[\lambda]}$ is a bimodule over $\she_\dTM = \she_\dTM^{[0]}$ and 
for any $P\in \she_\dTM(\lambda)$ and $Q\in \she_\dTM(\mu)$ one has
$$
\sigma_{\lambda +\mu}(PQ)=\sigma_\lambda(P)\sigma_\mu(Q).
$$
Recall that a microdifferential operator is invertible at $p\in\dTM$ if and only if its
principal symbol does not vanish at $p$.

\subsection{Automorphisms of $\she_\dTM$}

\begin{lemma}\label{lem:ffilt}
Any $\C$-algebra automorphism of $\she_\dTM$ is filtered and symbol
preserving.
\end{lemma}

\begin{proof}
Let $f$ be a $\C$-algebra automorphism of $\she_\dTM$. 
Define the spectrum of $P\in \sect(V;\she_\dTM)$ as
\begin{align*}
\Sigma(P)\colon V &\to \mathcal P(\C)\\
p &\mapsto \{a\in\C \colon a-P \text{ is not invertible at }p\},
\end{align*}
where $\mathcal P(\C)$ denotes the set of subsets of $\C$.
Note that $\Sigma(P) = \Sigma(f(P))$.
Set for short
\[
\she_m = \she_\dTM(m)\setminus\she_\dTM(m-1).
\]
Recall that $P$ is invertible if and only if its principal symbol does not vanish. 

\smallskip\noindent(i)
If $P\in\she_0$ and its principal symbol is not locally constant,
then $\Sigma(P)(p) = \{\sigma(P)(p)\}$.
Since $\Sigma(P) = \Sigma(f(P))$, it follows that $f(P)\in\she_0$
and $\sigma(P) = \sigma(f(P))$.

\smallskip\noindent(ii)
Let $P\in\she_0$ have locally constant principal symbol.
For any
$Q\in\she_\dTM(0)\setminus\opb{\sigma_0}(\C_\dTM)$ one has
\[
\begin{split}
\sigma_0(P)\sigma_0(Q) 
= \sigma_0(PQ)
= \sigma_0(f(PQ)) 
= \sigma(f(P))\sigma_0(f(Q)) 
= \sigma(f(P))\sigma_0(Q)
\end{split}
\]
where the second equality follows from (i).
One deduces $\sigma(f(P)) = \sigma_0(P)$, so that in particular $f(P)\in\she_0$.

\smallskip\noindent(iii)
Pick an operator $D\in \she_1$ invertible at $p$, and let $d$ be the order of
$f(D)$. Then $f(D)^m$ is an invertible operator of order $d m$ and one has
\[
f(\she_m)
= f (D^m \she_0 ) 
= f(D)^m f(\she_0) 
= f(D)^m \she_0 
= \she_{d m}.
\]
Since $f$ is an automorphism of $\she_\dTM \setminus \{0\} =
\DUnion\nolimits_{m\in\Z } \she_m$, it follows that $d=\pm 1$. Thus $f$ either
preserves or reverses the order. Note that if an operator $P$ satisfies
$\sigma(P)(p)=0$, then $\Sigma(P)(p)=\C$ if and only if $P$ has
positive order. Hence, $f$ preserves the order.

\smallskip\noindent(iv)
We have proved that $f$ is filtered and preserves the symbol of operators in
$\she_0$. As $\she_m = D^m \she_0$, to show that $f$ is symbol preserving it is enough
to check that
$\sigma_1(D) = \sigma_1 (f(D)).$

Let $(x;\xi)$ be a local system of symplectic coordinates at $p$. Identifying
$x_i$ with the operator in $\she_0$ whose total symbol is $x_i$, one has
\[
\begin{split}
\partial_{\xi_i}\sigma_1(D)
&= \{x_i, \sigma_1(D)\}
= \{\sigma_0(x_i), \sigma_1(D)\}
= \sigma_0([x_i,D])  \\
&= \sigma_0(f([x_i,D])) 
= \sigma_0([f(x_i),f(D)]) 
= \{ \sigma_0(f(x_i)), \sigma_1(f(D))\} \\
&= \{ x_i, \sigma_1(f(D))\} =
\partial_{\xi_i}\sigma_1f((D)),
\quad\text{for }i=1,\dots,n,
\end{split}
\]
so that
\[
\sigma_1(D) = \sigma_1 (f(D)) + \varphi(x),
\]
and one takes the homogeneous component of degree $1$. 
\end{proof}

\begin{proposition}\label{pro:fadp}
Any $\C$-algebra automorphism of $\she_\dTM$ is locally of the form $\ad(P)$
for some $\lambda\in \C$ and some invertible $P\in \she_\dTM(\lambda)$.
\end{proposition}

\begin{proof}
Identify $\dTM\times\dTM$ to an open subset of $T^*(M\times M)$.
Let $(x)$ be a system of local coordinates on $M$, and denote by $(x,y)$ the
coordinates on $M\times M$. For $Q\in\she_\dTM$, denote by $Q_x$ and $Q_y$ its
pull-back to $\she_{\dTM\times \dTM}$ by the first and second projection,
respectively.

Let $f\colon \she_\dTM \to \she_\dTM$ be a $\C$-algebra automorphism. 
By Lemma~\ref{lem:ffilt}, $f$ is filtered and symbol preserving. 
Denote by $\shl$ the $\she_{\dTM\times \dTM}$-module with one generator $u$ and relations
\[
\big(x_i-f(y_i) \big) \,u= \big(\partial_{x_i} - f(\partial_{y_i}) \big) \,u =0, 
\quad\text{for }i=1,\dots,n. 
\]
Then the image $f(Q)$ of $Q\in\she_\dTM$ is characterized by the relation 
\begin{equation}\label{eq:fQ}
f(Q)_y\, u = Q^*_x \, u 
\quad\text{in }\shl,
\end{equation}
where $Q^*$ denotes the adjoint operator, and 
$(\shl,u)$ is a simple module along the conormal bundle to the diagonal
$\Delta$ in $T^*(M\times M)$ (see~\cite{Kas03}).
Denote by $\shc_\Delta$ the sheaf of complex microfunctions along the conormal bundle to 
$\Delta$. By \cite[Theorem~8.21]{Kas03}, there exist $\lambda\in\C$ and an
isomorphism
\[
\varphi\colon \she_{\dTM\times\dTM}^{[\lambda]} \tens[\she_{\dTM\times\dTM}] \shc_\Delta \isoto \shl,
\] 
so that $\varphi(P_y\tens\delta_\Delta) =u$ for some invertible $P\in
\she_\dTM(\lambda)$. One then has
\[
\begin{split}
P_yQ_yP^{-1}_yu 
&=P_yQ_yP^{-1}_y\varphi(P_y\tens\delta_\Delta)
=\varphi(P_yQ_y\tens\delta_\Delta)
=\varphi(Q^*_xP^*_x\tens\delta_\Delta) \\
&=Q^*_x\varphi(P^*_x\tens\delta_\Delta)
=Q^*_x\varphi(P_y\tens\delta_\Delta)
= Q^*_xu.
\end{split}
\]
It follows by \eqref{eq:fQ} that one has $f = \ad(P)$.
\end{proof}

\subsection{Invertible $\she$-bimodules}

Denote by $P^*M$ the projective cotangent bundle of $M$ and by $\gamma\colon \dTM \to P^*M$ the projection. Set
$$
\she_{P^*M} = \oim\gamma\she_{\dTM}.
$$
This is a sheaf of $\C$-algebras endowed with a $\Z$-filtration such that 
$\gr\she_{P^*M}\simeq \DSum_{m\in\Z}\sho_{P^*M}(m)$, where one sets 
$\sho_{P^*M}(m)=\oim\gamma\sho_{\dTM}(m)$. Note that $\she_{\dTM}$ is constant along 
the fibers of $\gamma$. Since these are connected, the adjunction morphism gives an isomorphism
$$
\opb \gamma \she_{P^*M} \isoto \she_{\dTM}.
$$

\begin{lemma}\label{lem:extE}
Let $Z\subset \dTM$ be a closed conic analytic subset. Then
$$
H^j\rsect_Z\she_\dTM = 0 \qquad \text{for $j<\codim_\dTM Z$.}
$$
\end{lemma}

\begin{proof}
Setting $W=\gamma(Z)$, we have $\rsect_Z\she_\dTM \simeq
\opb\gamma\rsect_W\she_\PM$.
We thus have to show that $H^j\rsect_W\she_\PM = 0$ for
$j<\codim_\PM W$. Identify $\she_\PM$ with the sheaf $\shc_\Delta$ of 
complex microfunctions along the conormal bundle to the diagonal in $P^*=P^*(M\times M)$.
By quantized contact transformations, $\shc_\Delta$ can further be identified
with the sheaf of complex microfunctions $\shc_S$ along the conormal bundle to
a hypersurface $S\subset P^*$. One has $\shc_S \simeq
\sho_S\dsum H^1_{[S]}\sho_{P^*} \simeq \sho_S^{\oplus\Z}$. Hence $H^j\rsect_W\shc_S
= 0$ for $j<\codim_SW$.
\end{proof}

\begin{proposition}
\label{pr:M**}
Let $\shm$ be a coherent torsion-free $\she_\dTM$-module. Then $\shm$ is locally
free outside a closed conic analytic $2$-codimensional subset.
\end{proposition}

\begin{proof}
We will reduce to the analogous statement for $\sho$-modules, which is well-known 
(see \cite[Corollary 5.15]{Kob87}).

Set for short $\she=\she_\dTM$, $\she(0)=\she_\dTM(0)$ and $\sho(0)=\sho_\dTM(0)$. 
A coherent $\she(0)$-submodule $\shl\subset\shm$ such that $\she\shl=\shm$ is called a lattice.

(a) $\shm$ has a torsion-free lattice $\shl$. In fact, let $\shf$ be a lattice
in $\shm^* = \shHom[\she](\shm,\she)$. Then $\shf^* =
\shHom[\she(0)](\shf,\she(0))\subset \shHom[\she](\shm^*,\she) = \shm^{**}$ and
$\she\shf^*=\shm^{**}$, i.e.~$\shf^*$ is a lattice in $\shm^{**}$. Then
$\shl=\shf^*\cap\shm$ is a lattice in $\shm$. Since $\shf^{*}$ is reflexive (that is, $\shf^{*} \to (\shf^*)^{**}$ is an isomorphism), $\shf^*$ is torsion free, and so is its submodule $\shl$.

(b) The coherent $\sho(0)$-module $\overline\shl = \shl/\shl(-1)$ is torsion-free. In fact, 
consider the exact sequence
\[
0 \to \she(-1) \to \she(0) \to[\sigma_0] \sho(0) \to 0.
\]
Then $\sho(0) \tens[\she(0)] \shl \simeq \overline\shl$. Hence $(\overline\shl)^{*} =
\shHom[\sho(0)](\overline\shl,\sho(0)) \simeq \shHom[\sho(0)](\sho(0) \tens[\she(0)] \shl,\sho(0)) \simeq
\shHom[\she(0)](\shl,\sho(0))$. The exact sequence
\[
0 \to \shHom[\she(0)](\shl,\she(-1) ) \to \shHom[\she(0)](\shl,\she(0) ) \to
\shHom[\she(0)](\shl,\sho(0))
\]
thus reads
\[
0 \to \shl^*(-1) \to \shl^* \to (\overline\shl)^{*}.
\]
Hence $\overline{\shl^*} \subset (\overline\shl)^{*}$. Then
$\overline\shl\subset\overline{\shl^{**}}\subset (\overline{\shl^*})^* \isoto
(\overline{\shl^*})^{***}$, so that $\overline\shl$ is torsion-free.

(c) Since $\overline\shl$ is torsion-free, it is locally free outside a
closed conic analytic $2$-codimensional subset $S$. Hence the same holds true for $\shl$ by Nakayama
lemma. Thus $\shm=\she\shl$ is also locally free outside $S$.
\end{proof}

\begin{remark}
\label{rem:tf}
Since projective $\she_\dTM$-modules are torsion-free, it follows that $\she_\dTM$ is (coherent)
projective-free if $\dim M = 1$. This is no more true if $\dim M > 1$.
\end{remark}

Set
\[
\she_\dTM^e = \she_\dTM^\op\tens[\C]\she_\dTM.
\]
Note that, for $[\lambda],[\mu]\in\C/\Z$  the morphism of $\she_\dTM^e$-modules
$$
\she_\dTM^{[\lambda]}\tens[\she_\dTM]\she_\dTM^{[\mu]} \to \she_\dTM^{[\lambda + \mu]}, \quad 
P\tens Q \mapsto PQ
$$
is an isomorphism. In particular, $\she_\dTM^{[\lambda]}$ is an invertible $\she_\dTM^e$-module.
Moreover, if $P\in\she_\dTM(\lambda)$ has non-vanishing symbol on $V\subset \dTM$, 
there is an isomorphism of $\she_V^e$-modules (where we use notation~\eqref{eq:2modalg})
\begin{equation}
\label{eq:Elloc}
{}_{\ad(\opb P)}\she_V \isoto \she_V^{[\lambda]}, \quad
Q \mapsto PQ.
\end{equation}

\begin{lemma}\label{lem:Elm}
For $[\lambda],[\mu]\in\C/\Z$, one has
\[
\shHom[\she_\dTM^e](\she_\dTM^{[\lambda]},\she_\dTM^{[\mu]}) =
\begin{cases}
\C_\dTM &\text{for }[\lambda] = [\mu], \\
0 &\text{otherwise}.
\end{cases}
\]
\end{lemma}

\begin{proof}
The problem is local and we take a system $(x) = (x_1,\dots,x_n)$ of
local coordinates in $V\subset \dTM$ such that $\partial_1$ is invertible in $V$. 
By~\eqref{eq:Elloc}
\begin{align*}
\shHom[\she_{V}^e](\she_{V}^{[\lambda]},\she_{V}^{[\mu]}) 
&\simeq \shHom[\she_{V}^e]({}_{\ad(\partial_1^{-\lambda})}\she_{V},
{}_{\ad(\partial_1^{-\mu})} \she_{V}) \\
&\simeq \{ P\in\she_{V} \colon P \partial_1^{-\lambda} Q \partial_1^{\lambda} = 
\partial_1^{-\mu} Q \partial_1^{\mu} P,\ \forall Q\in\she_{V} \}.
\end{align*}
Assume that there exists $P\neq 0$ as above.
Taking for $Q$ the operators $\partial_1$, $x_i$ and $\partial_i$,
respectively, we deduce that $[P,\partial_1] = [P,x_i] = [P,\partial_i] = 0$
for $i=2,\dots,n$. It follows that $P$ only depends on $\partial_1$. Noting
that $[\partial_1^\lambda,x_1] = \lambda\partial_1^{\lambda-1}$ and taking
$Q=x_1$, we get
\[
[x_1,P] = (\mu - \lambda) P \partial_1^{-1}.
\]
Write $P = \sum_{j\leq m}c_j
\partial_1^j$ with $c_i\in\C$ and $c_m\neq 0$. Then the above equality gives $m=\mu -
\lambda$ and $c_j=0$ for $j<m$.
\end{proof}

The following result was communicated to us by Masaki Kashiwara (refer to \cite{KV10} for related results).

\begin{theorem}\label{thm:Eeloc}
Any invertible $\she_\dTM^e$-module is isomorphic to
$L\tens[\C]\she_\dTM^{[\lambda]}$, for some local system of rank one $L$ and some locally constant $\C/\Z$-valued function $[\lambda]$.
\end{theorem}

\begin{proof}
Set for short $\she=\she_\dTM$. Let $\shp$ be an invertible $\she^e$-module. 
It is enough to show that $\shp$ is locally  isomorphic to
$\she^{[\lambda]}$ for some locally constant function $[\lambda]$.
In fact, it will follow from Lemma~\ref{lem:Elm} that 
$L = \shHom[\she^e](\she^{[\lambda]},\shp)$ is a local system of rank one and
$L\tens[\C]\she^{[\lambda]} \isoto \shp$.

(a) Since $\shp$ is invertible, the underlying $\she$-module ${}_{\bullet}\shp$ is projective locally of finite presentation by (iv) and (v)  in Section~\ref{sse:Morita}, and hence coherent torsion-free. By Proposition~\ref{pr:M**}, ${}_{\bullet}\shp$ is locally free outside a closed analytic $2$-codimensional subset $Z$. As $\shp$ is invertible, its rank is one.

(b) Suppose that ${}_{\bullet}\shp$ is free of rank one. Then there exists 
$[\lambda]$ such that $\shp^{[-\lambda]} = \shp \tens[\she^e] \she^{[-\lambda]}$ admits a regular generator, i.e.~a generator $u$ of ${}_{\bullet}\shp^{[-\lambda]}$ such that $Pu = uP$ for any $P\in\she$. Indeed, let $t$ be a generator of ${}_{\bullet}\shp$ and let $f\colon\she\isoto \she$, 
be the $\C$-algebra isomorphism as in ~\eqref{eq:LocFreeBim}: $f(P) = Q$ for $Q$ such that $tP=Qt$.
By Proposition~\ref{pro:fadp}, $f$ is locally of the form $\ad(P)$ for some $\lambda\in \C$ and 
$P\in \she(\lambda)$ with never vanishing symbol. Then $u=t\opb P$ is a regular generator of $\shp^{[-\lambda]}$.

Let $V$ be a contractible open neighborhood of a point in $Z$.
We are left to show that if ${}_{\bullet}\shp$ is locally free of rank one on 
$V\setminus Z$, then ${}_{\bullet}\shp^{[-\lambda]}$ has a regular generator on $V$. 
It will follow that $\shp|_{V}\simeq \she^{[\lambda]}_{V}$.

(c) Since local regular generators $u$ of $\shp^{[-\lambda]}$ are unique up to multiplicative 
constants, $\C u \subset \shp^{[-\lambda]}$ defines a local system of rank one on $V\setminus Z$.
As $V\setminus Z$ is simply connected, such local system is constant. 
Thus $\shp^{[-\lambda]}$ has a regular generator $u$ on $V\setminus Z$. 

Consider the distinguished triangle
$$
\rsect_Z \shp^{[-\lambda]} \to \shp^{[-\lambda]} \to \rsect_{V\setminus Z}\shp^{[-\lambda]} \to[+1]
$$
Since $\shp^{[-\lambda]}$ is invertible, then ${}_{\bullet}\shp^{[-\lambda]}$ is flat by (iv)  in Section~\ref{sse:Morita}, so that 
$$\rsect_Z(V;\shp^{[-\lambda]}) \simeq \rsect(V;\rsect_Z\she\tens[\she]\shp^{[-\lambda]}).$$
By Lemma~\ref{lem:extE} one gets $H^j\rsect_Z(V;\shp^{[-\lambda]})  = 0$ for $j=0,1$. 
It follows that $\sect(V;\shp^{[-\lambda]})\isoto\sect(V\setminus Z;\shp^{[-\lambda]})$, hence the generator $u$ of ${}_{\bullet}\shp^{[-\lambda]}$ on $V\setminus Z$ extends uniquely to $V$.

\end{proof}
In particular, since any $\she_\dTM^{[\lambda]}$ is a locally free  right 
$\she_\dTM$-module of rank one  by~\eqref{eq:Elloc}, it follows that the $\C$-algebra $\she_\dTM$ is Picard good.

Recall that the projection  $\gamma\colon \dTM \to P^*M$ is a principal $\C^\times$-bundle. 

\begin{theorem}\label{thm:localEPicardGood}
The $\C$-algebra $\she_{P^*M}$ is Picard good.
\end{theorem}

\begin{proof}
Let us prove that any invertible $\she_{P^*M}^e$-module $\shp$
is locally free of rank one as right $\she_{P^*M}$-module.

Since this is a local problem, we may restrict to a contractible open subset $U\subset P^*M$, so that $\opb\gamma(U) \simeq U\times\C^\times$. The $\she_{\opb\gamma(U)}^e$-module $\opb \gamma \shp$ being invertible, by Theorem~\ref{thm:Eeloc} one gets
\[
\shp \isoto \oim\gamma\opb\gamma\shp \simeq \oim\gamma(L\tens[\C] \she_{\opb\gamma(U)}^{[\lambda]})
\]
for some $[\lambda]\in\C/\Z$ and some local system of rank one $L$ on $\opb\gamma(U)$ with monodromy $e^{-2 \pi i \lambda}$ on $\C^\times$. 

By restricting to $U'\subset U$, we may assume that there exists an invertible operator $D$ of order 1. This defines an isomorphism of right $\she_{U'}$-modules
$$
\she_{U'} \isoto \oim\gamma(L\tens[\C]
\she_{\opb\gamma(U')}^{[\lambda]})  \quad
Q \mapsto  D^\lambda Q.
$$
\end{proof}
Note that, given a local system of rank one $L$ and $[\lambda]\in\C/\Z$, one has 
$\oim \gamma (L \tens[\C] \she_\dTM^{[\lambda]})\neq 0$ if and only if the monodromy of $L$ 
along the fiber of $\gamma$ is given by $e^{-2 \pi i\lambda}$.
In particular, $\oim \gamma \she_{\dTM}^{[\lambda]}= 0$ for any $[\lambda]\neq 0$.

\section{Microdifferential algebroids}\label{se:results}

Here we state and prove our results on classification of
$\she$-algebroids on a contact manifold.

\subsection{Contact manifolds}
Let $X$ be a complex manifold of odd dimension, say $2n-1$. 
Denote by $\sho_X$ the sheaf of
holomorphic functions and by $\Omega^1_X$ the sheaf of holomorphic $1$-forms.
A structure of
(complex) contact manifold on $X$ is the assignment of a holomorphic principal
$\C^\times$-bundle $\gamma\colon Y\to X$, called symplectification, and of a
holomorphic one-form $\alpha \in \sect(Y;\Omega^1_Y)$, called contact form,
such that $\omega = d \alpha$ is symplectic (i.e.~$\omega^n$ vanishes nowhere)
and
$i_\theta \alpha = 0$,
$L_\theta \alpha = \alpha$. 
Here, $\theta$ denotes the infinitesimal generator of the action of
$\C^\times$ on $Y$, $i_\theta$ the contraction and $L_\theta$ the Lie
derivative. 
One may consider $\alpha$ as a global section of
$\Omega^1_X\tens[\sho_X]\sho_X(1)$, where $\sho_X(1)$ denotes the dual
of the sheaf of sections of the line bundle $\C\times_{\C^\times}Y$.

Let $M$ be a complex manifold of dimension $n$.
Then $P^*M$ has a natural contact structure given by the Liouville one-form on
$\dTM$ and by the projection $\gamma\colon \dTM \to P^*M$.
By Darboux theorem, $P^*M$ is a local model for a contact manifold $X$, meaning that
there are an open cover $\{U_i\}_{i\in I}$ of $X$ and
contact embeddings (i.e.~embeddings preserving the contact forms) $j_i\colon
U_i \hookrightarrow P^*M$ for any $i\in I$.

A fundamental result by~\cite{S-K-K} asserts 
that contact transformations (i.e.~biholomorphisms
preserving the contact forms) can be locally quantized. This means the
following. Let $N$ be another complex manifold of dimension $n$, $U\subset
P^*M$ and $V\subset P^*N$ open subsets and $\chi\colon U\to V$ a contact
transformation. Then any $x\in U$ has an open neighborhood $U'$ such that
there is a $\C$-algebra isomorphism $\opb \chi (\she_{P^*N}|_{\chi (U')}) \isoto
\she_{P^*M}|_{U'}$.

\begin{definition}\label{def:mical}
An $\she$-algebra on a contact manifold $X$ is a sheaf
$\sha$ of $\C$-algebras such that there are an open cover $\{U_i\}_{i\in I}$
of $X$, contact embeddings $j_i\colon U_i \hookrightarrow P^*M$ and
$\C$-algebra isomorphisms $\sha|_{U_i} \simeq \opb j_i\she_{P^*M}$ for any 
$i\in I$.
\end{definition}

Given an $\she$-algebra $\sha$, the $\C$-algebra $\opb \gamma \sha$ on $Y$
satisfies $\opb \gamma \sha|_{\opb \gamma(U_i)} \simeq \opb {\tilde j_i}\she_{\dTM}$ for 
$\tilde j_i\colon \opb \gamma(U_i) \to \dTM$ a homogeneous symplectic transformation lifting 
$j_i\colon U_i \hookrightarrow P^*M$. 
Note that Proposition~\ref{pro:fadp} implies that the invertible
$\opb \gamma \sha^e$-module $(\opb \gamma \sha)^{[\lambda]}$ is well-defined for any
$[\lambda]\in \C/\Z$.

To quantize $X$ in the strict sense means to endow it with an
$\she$-algebra (see~\cite{Bou99,Pol15}). This might not be possible in
general. However, as we now recall, Kashiwara~\cite{Kas96} proved that 
$X$ is endowed with a canonical $\she$-algebroid.

\subsection{Microdifferential algebroids}

\begin{definition}\label{def:microalg}
\begin{itemize}
\item[(i)] 
An $\she$-algebroid on $X$ is a $\C$-algebroid $\stka$
such that for every open subset $U\subset X$ and any object
$\obja\in\stka(U)$, the $\C$-algebra $\shEnd[\stka](\obja)$ is an
$\she$-algebra on $U$.
\item[(ii)] 
A stack of twisted $\she$-modules on $X$ is a $\C$-stack $\stkm$ such that
there are an open cover $\{U_i\}_{i\in I}$ of $X$, $\she$-algebras
$\she_i$ on $U_i$ and equivalences $\stkm|_{U_i} \equi[\C] \stkMod(\she_i)$
for any $i\in I$.
\end{itemize}
\end{definition}

Note that a $\C$-stack $\stka$ is an $\she$-algebroid if and
only if there are an open cover $\{U_i\}_{i\in I}$ of $X$, $\she$-algebras
$\she_i$ on $U_i$ and equivalences $\stka|_{U_i}\equi[\C]\astk{\she_i}$ for any $i\in I$. 
In particular, $\stkMod(\stka)$ is a stack of twisted $\she$-modules.

Kashiwara's construction of the canonical $\she$-algebroid on $X$ was performed by patching data as explained in
Appendix~\ref{se:acocy} (see~\cite{DK11} for a more intrinsic
construction). 
More precisely, in~\cite{Kas96} he proved the
existence of an open cover $\shu = \{U_i\}_{i\in I}$ of $X$, of
$\she$-algebras $\she_i$ on $U_i$, of
isomorphisms of $\C$-algebras $f_{ij}\colon \she_j\to\she_i$
on $U_{ij}$ and of sections
$a_{ijk}\in\sect(U_{ijk};\she_i(0)^\times)$, satisfying the cocycle condition
\begin{equation}
\label{eq:KasCoc}
\begin{cases}
f_{ij}f_{jk} = \ad(a_{ijk})f_{ik},\\ 
a_{ijk} a_{ikl} = f_{ij}(a_{jkl}) a_{ijl}.
\end{cases}
\end{equation}
By Proposition~\ref{pr:glue}~(i), this implies

\begin{theorem}[\cite{Kas96}]
Any complex contact manifold $X$ is endowed with a canonical
$\she$-algebroid $\stke_X$.
\end{theorem}
It follows that a $\C$-stack on $X$ is an
$\she$-algebroid (resp.~a stack of twisted $\she$-modules) if and only if it 
is locally $\C$-equivalent to $\stke_X$ (resp.~to $\stkMod(\stke_X)$). In particular, if 
$X=P^*M$ then  $\stke_{P^*M}$ is $\C$-equivalent to $\she_{P^*M}$, and $\she$-algebroids 
are $\C$-twisted forms of $\she_{P^*M}$.

Recall that an algebroid is Picard good if and only if  so are the algebras 
that locally represent it. Hence, by Theorem~\ref{thm:localEPicardGood} one gets 
that any $\she$-algebroid, and in particular $\stke_X$, is Picard good. From Proposition~\ref{pro:picgood}, 
we thus deduce the following

\begin{theorem}\label{thm:Emoritatrivial}
\begin{itemize}
\item[(i)] 
Two $\she$-algebroids are $\C$-equivalent if and only
if they are Morita equivalent.
\item[(ii)] 
Any stack of twisted $\she$-modules is $\C$-equivalent to the stack of modules
over an $\she$-algebroid.
\end{itemize}
\end{theorem}

To classify $\she$-algebroids, we thus need to compute the first cohomology
with value in the stack of 2-groups $\stkAut[\C](\stke_X)\equi[]\stkInv(\stke_X^e)_\op$,
where we set $\stke_X^e = \stke_X^\op\tens[\C]\stke_X$.

\subsection{Geometry of $\gamma\colon Y\to X$}\label{se:gamma}

\begin{lemma}\label{lem:roimgamma}
For $M$ an abelian group, there is a distinguished triangle
$$
M_X \to \roim\gamma M_Y \to M_X[-1] \to[+1]
$$
\end{lemma}

\begin{proof}
As the complex $\roim\gamma M_Y$ is concentrated in degrees $[0,1]$, by
truncation it is enough to prove the isomorphisms
\begin{equation}\label{eq:roimgammatemp}
H^i\roim\gamma M_Y \simeq M_X,\quad\text{for } i=0,1.
\end{equation}
For $i=0$ it is induced by the adjunction morphism $M_X \to
\roim\gamma M_Y$.

Set $SY=Y/\R_{>0}$ and consider $\gamma$ as the composite of $p\colon Y \to
SY$ and $q\colon SY\to X$, which are principal bundle for the groups $\R_{>0}$
and $S^1$, respectively. Note that $\roim p M_Y \simeq M_{SY}$, so that
$\roim\gamma M_Y \simeq \roim q M_{SY} \simeq \reim q M_{SY}$. The
infinitesimal generator $\theta$ of the action of $\C^\times$ on $Y$ induces a
trivialization of the relative orientation sheaf $or_{SY/X}$. Hence $\epb q
M_X \simeq M_{SY}[1]$. Then the isomorphism \eqref{eq:roimgammatemp} for $i=1$
is induced by the adjunction morphism $\reim q M_{SY} \simeq \reim q \epb q
M_X[-1] \to M_X[-1]$.
\end{proof}

Let $M=\C^\times$. The induced long exact cohomology sequence gives:
\[
H^1(Y;\C^\times) \to[\mu_1] H^0(X;\C^\times) \to[\delta]
H^2(X;\C^\times) \to[\gamma^\#] H^2(Y;\C^\times) \to[\mu_2] H^1(X;\C^\times) .
\]
We can represent elements of $H^0$ by locally constant $\C^\times$-valued functions, 
elements of $H^1$ by isomorphism classes of local systems of rank one, and
elements of $H^2$ by $\C$-equivalence classes of invertible $\C$-algebroids
(see Lemma~\ref{lem:tors}).
Let us describe the above sequence in these terms (see also~\cite[Chapitre V \S\S 3.1,3.2]{Gir71}),
were we use the notation $[\,\cdot\,]$ both for isomorphism and $\C$-equivalence classes.

For $L$ a local systems of rank one on $Y$, $\mu_1([L])$ is the locally constant function on $X$ giving the monodromy of $L$ along the fibers of $\gamma$. 

Recall that $\astk{\C_Y}$ denotes the stack of local systems of rank one on $Y$.

\begin{lemma}\label{lem:pi_0} 
\begin{itemize}
\item[(i)]
There is a group isomorphism $\pi_0(\oim\gamma\astk{\C_Y})\simeq\C_X^\times$, where the group structure on the left-hand side is induced by $\tens_\C$.
\item[(ii)]
If $\stkd$ is a $\C$-stack on $Y$, then $\pi_0(\oim \gamma \stkd)$ is a $\C_X^\times$-sheaf (i.e., it is endowed with a $\C_X^\times$-action).
\item[(iii)]
If $\stkt$ is an invertible $\C_Y$-algebroid, then $\pi_0(\oim \gamma \stkt)$ is a $\C_X^\times$-torsor.
\end{itemize}
\end{lemma}

\begin{proof}
(i) Recall that $\C^\times_Y[1]$ denotes the stack of $\C_Y^\times$-torsors. The functor
\[
\C_Y^\times[1] \to\astk{\C_Y},
\quad \shp \mapsto \C \times_{\C^\times} \shp
\]
defines a group isomorphism $\pi_0(\oim\gamma\astk{\C_Y})\simeq \pi_0(\oim\gamma(\C^\times_Y[1]))$.
By~\eqref{eq:pi-roim}, the latter is isomorphic to $R^1\oim \gamma \C_Y^\times$, hence to $\C_X^\times$ by Lemma~\ref{lem:roimgamma}.

\smallskip\noindent
(ii) By using~\eqref{eq:dir-im}, one gets a $\C$-functor 
$$
\oim\gamma\astk{\C_Y}\tens[\C] \oim\gamma\stkd \to \oim\gamma\stkd, \quad (L,\delta) \mapsto L\tens[\C]\delta.
$$
This defines an action of  $\pi_0(\oim\gamma\astk{\C_Y})\simeq\C_X^\times$ on $\pi_0(\oim\gamma\stkd)$.

\smallskip\noindent
(iii) Since $R^2\oim \gamma \C_Y^\times = 0$, the stack $\oim\gamma\stkt$ is locally $\C$-equivalent to $\oim\gamma\astk{\C_Y}$. Hence $\pi_0(\oim\gamma\stkt)$ 
is a $\C_X^\times$-torsor by (i) and (ii).
\end{proof}

\begin{notation}\label{nt:subAlg}
Let $\stkc$ be a $\C$-stack on $X$. For $s$ a global section of $\pi_0(\stkc)$, we denote by $\stkc^s$ the full substack of $\stkc$ whose objects $c$ satisfy $[c] = s$ in $\pi_0(\stkc)$.
\end{notation}
Note that $\stkc^s$ is a $\C$-algebroid, since $\pi_0(\stkc^s)=\{s\}_X$. It is locally $\C$-equivalent to the algebra $\shEnd[\stkc](c)$ for any local representative $c$ of $s$.

\medskip

For $m \in H^0(X;\C^\times) \simeq \sect(X;\pi_0(\oim\gamma\astk{\C_Y}))$, one has
$$\delta(m) = [(\oim\gamma\astk{\C_Y})^m].$$
Here, $(\oim\gamma\astk{\C_Y})^m$ is identified with the $\C_X$-algebroid of local systems $L\in\oim\gamma\astk{\C_Y}$ with $\mu_1([L]) = m$. In particular, for $m=1$ the $\C_X$-algebroid $(\oim\gamma\astk{\C_Y})^1$ is equivalent to 
$\astk{\C_X}$ via the adjunction functor $\astk{\C_X} \to \oim \gamma \astk{\C_Y}$. Moreover, one has a decomposition $\oim\gamma\astk{\C_Y} \equi[\C]
\coprod\nolimits_{m\in\C_X^\times} (\oim\gamma\astk{\C_Y})^m$.

For $\stks$ an invertible $\C_X$-algebroid, $\gamma^\#([\stks]) =[\opb\gamma\stks]$. 

\begin{proposition}
\label{pro:mu2}
For $\stkt$ an invertible $\C_Y$-algebroid, $\mu_2([\stkt])$ is the class of the local system of rank one 
$\C \times_{\C^\times} \pi_0(\oim\gamma\stkt)$.
\end{proposition}

\begin{proof}
By Lemma~\ref{lem:pi_0}~(iii), $\pi_0(\oim\gamma\stkt)$ 
is a $\C_X^\times$-torsor.
It follows that $\C \times_{\C^\times} \pi_0(\oim\gamma\stkt)$ is a local system of rank one on $X$.

Choose an open covering $\{U_i\}$ of $X$ in such a way that $\stkt$ is described, by means of  Proposition~\ref{pr:patch} (i), by the data 
$(\astk \C_{V_i}, (\cdot)\tens[\C]M_{ji}, \transfa_{ijk})$, where  $V_i=\opb \gamma (U_i)$ and 
$M_{ji}$ are local systems of rank one on $V_{ij}$. Then $\C \times_{\C^\times} \pi_0(\oim\gamma\stkt)$ 
is represented by the 1-cocycle $\{\mu_1([M_{ji}])\}$ with values in $\C^\times$, which 
gives a Cech representative of the class $\mu_2([\stkt])$.
\end{proof}

\subsection{Classification results}

Set
\[
\stke_Y = \opb\gamma\stke_X, \quad \stke_Y^e = \stke_Y^\op\tens[\C]\stke_Y \equi[] \opb\gamma(\stke_X^e).
\]
Note that $\stke_Y$ can be described by patching the $\C$-algebras $\opb\gamma\she_i$ along the pull back 
on $Y$ of the data~\eqref{eq:KasCoc}.

For $[\lambda]\in\C/\Z$, the algebroid version of the invertible bimodule $\she_\dTM^{[\lambda]}$ is the $\stke_Y^e$-module
$$
\stke_Y^{[\lambda]} \in \stkFun[\C_Y](\stke_Y,\stke_Y) \subset \stkMod( \stke_Y^e)
$$
locally defined by $(\cdot)\tens[\she_\dTM]\she_\dTM^{[\lambda]}$
(cp. Proposition~\ref{pro:AutInv}).


Consider the direct image functor, obtained by using~\eqref{eq:2adj},
\[
\oim\gamma \colon \oim \gamma\stkMod(\stke_Y^e) \to \stkMod(\stke_X^e)
\]
and recall the morphism
$H^1(Y;\C^\times) \to[\mu_1] H^0(X;\C^\times)\simeq  H^0(X;\C/\Z)$ from \S\ref{se:gamma}.

\begin{theorem}\label{thm:gammainv}
The functor
\begin{equation}
\label{eq:E-inv}
\oim\gamma\stkInv(\C_Y) \to \stkInv(\stke_X^e),
\qquad L \mapsto \oim\gamma(L \tens[\C] \stke_Y^{\mu_1(L^*)})
\end{equation}is an equivalence of stacks of 2-groups. 
\end{theorem}

\begin{proof}
(a) A priori, $\oim\gamma(L \tens[\C] \stke_Y^{\mu_1(L^*)})$ is an object of
$\stkMod(\stke_X^e)$. It is locally, hence globally, invertible with inverse given by $\oim\gamma(L^* \tens[\C] \stke_Y^{\mu_1(L)})$.

(b) The sheaf $\C_Y$ is sent to $\stke_X$, since $\oim\gamma(\stke_Y)\simeq\stke_X$  as $\stke_X^e$-modules. Moreover, the natural morphism
$$
\oim\gamma(L \tens[\C] \stke_Y^{\mu_1(L^*)}) \tens[\stke_X] \oim\gamma(L'
\tens[\C] \stke_Y^{\mu_1(L'^*)}) \to \oim\gamma(L \tens[\C] L'^* \tens[\C] \stke_Y^{\mu_1(L^*) + \mu_1(L'^*)})
$$
is locally, hence globally, an isomorphism. Hence~\eqref{eq:E-inv} is monoidal.

(c) For an invertible $\stke_X^e$-module $\shp$, define its exponent as the unique locally 
constant $\C/\Z$-valued function $\epsilon(\shp)$ on $X$ such that $\opb \gamma \shp$ is locally isomorphic to $\stke_Y^{\epsilon(\shp)}$ (this is well-defined by Theorem~\ref{thm:Eeloc}.). Then $\epsilon(\oim\gamma(L\tens[\C] \stke_Y^{\mu_1(L^*)}))=\mu_1(L^*)$, and by using Lemma~\ref{lem:Elm} one gets that the functor 
$$
\shp\mapsto \hom[\stke_Y^e](\stke_Y^{\epsilon(\shp)}, \opb \gamma \shp)
$$
is a quasi-inverse of~\eqref{eq:E-inv}.
\end{proof}

Let $\mathrm{Pic}(\stke_X)$ denote the set of isomorphism class of invertible 
$\stke_X^e$-modules, endowed with the group structure induced by $\tens[\stke_X]$. 

\begin{corollary}\label{cor:class} There is a group isomorphism 
$\mathrm{Pic}(\stke_X)\simeq H^1(Y; \C_Y^\times).$
\end{corollary}

\begin{theorem}\label{thm:class}
The set of $\C$-equivalence classes (resp. Morita classes) of $\she$-algebroids is canonically isomorphic, as a pointed set, to $H^2(Y;\C^\times_Y)$.
\end{theorem}

\begin{proof}
Since $\stke_X$ is Picard good,  by Theorem~\ref{thm:gammainv} and Proposition~\ref{pro:Out} there is 
an equivalence of stacks of 2-groups
$$
\stkAut[\C](\stke_X) \equi[] \oim\gamma\stkInv(\C_Y)_\op.
$$
The right-hand term is equivalent to $\oim\gamma\stkInv(\C_Y)$ by the functor $L\mapsto L^*$.
Since $\C_Y$ is Picard good, from~\eqref{eq:TorsInv} and by using~\eqref{eq:roim} one gets an equivalence of stacks of 2-groups
\[
\oim\gamma\stkInv(\C_Y)  \equi[] [\roim\gamma\C_Y^\times[1]].
\]
It then follows from \eqref{eq:h1commcross} that
\begin{equation}
\label{eq:MoritaClass}
H^1(X; \stkAut[\C](\stke_X)) \simeq H^2(Y; \C_Y^\times).
\end{equation}
\end{proof}

We end by giving a realization of isomorphism~\eqref{eq:MoritaClass}.

\medskip

First, let us explain how to twist $\stke_Y$ by a local system of rank one $L$ on $X$, obtaining a $\C$-algebroid  $\stke_Y^L$ on $Y$ locally $\C$-equivalent to $\stke_Y$.

Choose an open covering $\{U_i\}$ of $X$ in such a way that $L$ is represented by a 1-cocycle 
$\{[\lambda_{ij}]\}$ with values in $\C/\Z$.  Set $V_i=\opb \gamma (U_i)$ and consider the data 
$(\stke_{V_i},  (\cdot)\tens[\stke_{V_{ij}}]\stke_{V_{ij}}^{[\lambda_{ij}]}, \transfm_{ijk})$, where 
$\transfm_{ijk}$ denotes the invertible transformation induced by the canonical isomorphism of 
$\stke_{V_{ijk}}^e$-modules 
$$
\stke_{V_{ijk}}^{[\lambda_{ij}]}  \tens[\stke_{V_{ijk}}]\stke_{V_{ijk}}^{[\lambda_{jk}]}\isoto \stke_{V_{ijk}}^{[\lambda_{ik}]}. 
$$
Then $\stke_Y^L$ is the $\C$-stack on $Y$ obtained from these data by Proposition~\ref{pr:patch} (i). 
Note that $(\stke_Y^L)^\op\equi[\C]\stke_Y^{L^*}$ and $\stke_Y^L\equi[\C]\stke_Y$ if $L$ is trivial.

Denote by $L^\times$ the $\C_X^\times$-torsor associated to $L$ and recall from Lemma~\ref{lem:pi_0} that  $\pi_0(\oim \gamma \stke_Y^L)$ is endowed with a $\C_X^\times$-action.

\begin{lemma}\label{lem:iso-pi_0L}
$\pi_0(\oim \gamma \stke_Y^L)\simeq L^\times \times_{\C^\times}\pi_0(\oim \gamma\stke_Y)$ as  
$\C_X^\times$-sheaves.
\end{lemma}

\begin{proof}
Let $\{[\lambda_{ij}]\}$ be a 1-cocycle with values in $\C/\Z$ representing $L$ on an open 
covering  $\{U_i\}$ of $X$. Then $\oim \gamma \stke_Y^L|_{U_i}\equi[\C]\oim \gamma \stke_Y|_{U_i}$ and the associated glueing $\C$-equivalences $\oim \gamma \stke_Y|_{U_{ij}}\to\oim \gamma \stke_Y|_{U_{ij}}$ are given by $(\cdot)\tens[\stke_{V_{ij}}]\stke_{V_{ij}}^{[\lambda_{ij}]}$, where $V_i=\opb \gamma(U_i)$. We thus get isomorphisms of  $\C^\times$-sheaves $\pi_0(\oim \gamma \stke_Y^L)|_{U_i}\simeq \pi_0(\oim \gamma\stke_Y)|_{U_i}$, with associated glueing automorphisms of
$\pi_0(\oim \gamma \stke_Y)|_{U_{ij}}$
given by multiplication by $e^{2\pi i \lambda_{ij}}$.
This follows from the commutative diagram of stacks of 2-groups
$$
\xymatrix{ \C/\Z_X[0] \ar[d]  \ar[r]^-{\simeq} & \C^\times_X [0] \ar[d]  \\
\oim\gamma\stkAut[\C](\stke_Y)   \ar[r]^-{\pi_0} & \shAut(\pi_0(\oim \gamma \stke_Y))[0],}$$
where the left-hand vertical arrow is the functor $[\lambda]\mapsto (\cdot)\tens[\stke_Y] \stke_Y^{[\lambda]}$ and the right-hand one is the $\C^\times$-action.
Hence $\pi_0(\oim \gamma \stke_Y^L)$ is isomorphic to $\pi_0(\oim \gamma \stke_Y)$ twisted by the 
$\C_X^\times$-torsor $L^\times$.
\end{proof}

Let $\stkt$ be an invertible $\C_Y$-algebroid. 
Following Proposition~\ref{pro:mu2}, we denote by $\mu_2(\stkt)$ the local system of rank one on $X$ associated to the $\C_X^\times$-torsor $\pi_0(\oim\gamma\stkt)$.

\begin{lemma}\label{lem:iso-pi_0T}
$\pi_0(\oim \gamma(\stkt\tens[\C]\stke_Y^{\mu_2(\stkt^\op)}))\simeq\pi_0(\oim \gamma\stke_Y)$ 
as $\C_X^\times$-sheaves.
\end{lemma}

\begin{proof}
By using the functor~\eqref{eq:dir-im}, one gets a morphism 
$$
\pi_0(\oim \gamma\stkt) \times \pi_0(\oim\gamma \stke_Y^{\mu_2(\stkt^\op)})\to 
\pi_0(\oim \gamma(\stkt\tens[\C]\stke_Y^{\mu_2(\stkt^\op)}))
$$ 
which is $\C^\times$-equivariant on each term. Hence it factors through $\pi_0(\oim \gamma\stkt) \times_{\C^\times} \pi_0(\oim\gamma \stke_Y^{\mu_2(\stkt^\op)})$. By Lemma~\ref{lem:iso-pi_0T}, this is isomorphic to $\pi_0(\oim \gamma\stke_Y)$, since $\pi_0(\oim \gamma\stkt^\op)$ is isomorphic to the $\C_X^\times$-torsor  opposite to $\pi_0(\oim \gamma\stkt)$. It follows that we have a morphism
$$
\pi_0(\oim \gamma\stke_Y)\to \pi_0(\oim \gamma(\stkt\tens[\C]\stke_Y^{\mu_2(\stkt^\op)}))
$$
of $\C_X^\times$-sheaves, which is locally, hence globally, an isomorphism.
\end{proof}

\begin{corollary}
$\pi_0(\oim \gamma(\stkt\tens[\C]\stke_Y^{\mu_2(\stkt^\op)}))$ has a canonical global section.
\end{corollary}

\begin{proof} 
The adjunction functor $\stke_X \to \oim \gamma \stke_Y$ defines a morphism 
$\pi_0(\stke_X) \to \pi_0(\oim \gamma \stke_Y)$. Since $\pi_0(\stke_X)$ is the singleton-valued constant sheaf, this gives 
a global section of $\pi_0(\oim \gamma \stke_Y)$, hence of $\pi_0(\oim \gamma(\stkt
\tens[\C]\stke_Y^{\mu_2(\stkt^\op)}))$ by Lemma~\ref{lem:iso-pi_0T}.
\end{proof}

Denote by  $can$ the canonical global section of $\pi_0(\oim \gamma(\stkt\tens[\C]\stke_Y^{\mu_2(\stkt^\op)}))$. Then the inverse of the isomorphism~\eqref{eq:MoritaClass} is realized as
\[
[\stkt] \mapsto [(\oim \gamma(\stkt\tens[\C]\stke_Y^{\mu_2(\stkt^\op)}))^{can}],
\]
where $[\, \cdot\, ]$ denotes the $\C$-equivalence class and we use the Notation~\ref{nt:subAlg}.

Assume that $\stkt = \opb\gamma\stks$, for $\stks$ an invertible $\C_X$-algebroid. Since $\mu_2(\opb\gamma\stks^\op)$ 
is trivial and $\stke_Y = \opb\gamma\stke_X$, the above isomorphism reduces to
\[
[\opb\gamma\stks] \mapsto [(\oim \gamma \opb\gamma(\stks\tens[\C]\stke_X))^{can}] = [\stks\tens[\C]\stke_X],
\] 
the latter being the class of the ``twist" of $\stke_X$ by $\stks$.

\begin{remark} Replacing $\stke_X$ by an $\she$-algebroid in the previous construction, one gets an action of $H^2(Y;\C^\times_Y)$ on the set of $\C$-equivalence classes (resp. Morita classes) of $\she$-algebroids. In such a way, the latter becomes an 
$H^2(Y;\C^\times_Y)$-torsor and the canonical isomorphism~\eqref{eq:MoritaClass} is obtained by choosing the $\C$-equivalence class of $\stke_X$ as base point.
\end{remark}

\appendix

\section{Cocycles}\label{se:coc}

For the reader's convenience, we recall here the descent conditions for stacks, and detail the case of algebroids. This is parallel to the case of gerbes. These results are well known and can be found for example in~\cite{Gir71,Bre94,Kas96} (see also~\cite{Kon01,DP04,PS04,DP05,BGNT07,Yek12}).

Let $X$ be a topological space (or a site), and $\shr$ a sheaf of commutative
rings on $X$.
If  $\shu=\{U_{i}\}_{i\in I}$ is an open cover of $X$, we set
$U_{ij} = U_{i} \cap U_{j}$, $U_{ijk} = U_{i} \cap U_{j} \cap U_{k}$ etc.
We use the notation $\bullet$ for the  horizontal composition of transformations.

\subsection{Glueing of stacks}\label{se:cocy} 
Let us recall here how to recover $\shr$-stacks, $\shr$-functors and transformations from collections of local data.

\begin{proposition}
\label{pr:patch}
Let $\shu = \{U_i\}_{i\in I}$ be an open cover of $X$.
\begin{itemize}
\item[(i)] 
Consider the descent datum $\triplet{\stkc_i}{\functf_{ij}}{\transfa_{ijk}}{ijk\in I}$,
that is, $\stkc_i$ are $\shr$-stacks on $U_i$, $\functf_{ij}\colon \stkc_j \to \stkc_i$
are $\shr$-equivalences on $U_{ij}$ and $\transfa_{ijk}\colon
\functf_{ik}\to\functf_{ij}\circ \functf_{jk}$ are invertible transformations on
$U_{ijk}$, such that
\begin{equation}
\label{eq:alpha}
\vcenter{\xymatrix@C=3em{ \functf_{ij} \circ \functf_{jk} \circ \functf_{kl} &
\functf_{ik} \circ \functf_{kl} \ar[l]^-{\transfa_{ijk}\bullet \id_{\functf_{kl}}} \\
\functf_{ij} \circ \functf_{jl} \ar[u]^{\id_{\functf_{ij}} \bullet \transfa_{jkl}} &
\functf_{il} \ar[u]_{\transfa_{ikl}} \ar[l]_{\transfa_{ijl}}
}} \qquad
\text{commutes.}
\end{equation}
Then there exists an $\shr$-stack $\stkc$ on $X$ endowed with
$\shr$-equivalences $\functf_i\colon \stkc|_{U_i}\to\stkc_i$ and invertible
transformations $\transfa_{ij} \colon \functf_i \to \functf_{ij} \circ
\functf_j$ on $U_{ij}$, such that
\[
\vcenter{\xymatrix@C=3em{ \functf_{ij} \circ \functf_{jk} \circ \functf_{k} &
\functf_{ij} \circ \functf_{j} \ar[l]^-{\id_{\functf_{ij}} \bullet \transfa_{jk}} \\
\functf_{ik} \circ \functf_{k} \ar[u]^{\transfa_{jkl} \bullet \id_{\functf_{k}}}&
\functf_{i} \ar[u]_{\transfa_{ij}} \ar[l]_{\transfa_{ik}} .
}} \qquad
\text{commutes.}
\]
The $\shr$-stack $\stkc$ is unique up to an $\shr$-equivalence unique up to a
unique invertible transformation.
\item[(ii)] 
Let $\stkc$ be as above, and let $\stkc'$ be associated with the descent datum \\
$\triplet{\stkc'_i}{\functf'_{ij}}{\transfa'_{ijk}}{ijk\in I}$. Consider the  descent datum
$\pair{\functg_i}{\transfb_{ij}}{ij\in I}$, that is, $\functg_{i}\colon
\stkc_i \to \stkc'_i$ are $\shr$-functors on $U_{i}$ and $\transfb_{ij}\colon
\functf'_{ij} \circ \functg_{j} \to \functg_{i} \circ \functf_{ij}$ are
invertible transformations on $U_{ij}$, such that
\begin{equation}
\label{eq:alpha2}
\vcenter{\xymatrix@C=3em{ \functg_{i} \circ \functf_{ij} \circ \functf_{jk} &
\functg_{i} \circ \functf_{ik} \ar[l]^-{\id_{\functg_{i}} \bullet \transfa_{ijk}} \\
\functf'_{ij} \circ \functg_{j} \circ \functf_{jk}
\ar[u]^{\transfb_{ij} \bullet \id_{\functf_{jk}}} &
\functf'_{ij} \circ \functf'_{jk} \circ \functg_{k}
\ar[l]_{\id_{\functf'_{ij}} \bullet \transfb_{jk}} & \functf'_{ik} \circ \functg_{k}
\ar[ul]_{\transfb_{ik}} \ar[l]_{\transfa'_{ijk} \bullet \id_{\functg_{k}}}
}} \qquad
\text{commutes.}
\end{equation}
Then there exists an $\shr$-functor $\functg\colon\stkc\to\stkc'$
endowed with invertible transformations $\transfb_{i} \colon \functf'_i \circ
\functg \to \functg_{i} \circ \functf_i$ on $U_{i}$, such that
\[ 
\vcenter{\xymatrix@C=3em{ \functg_{i} \circ \functf_{ij} \circ \functf_{j} &
\functg_{i} \circ \functf_{i} \ar[l]^-{\id_{\functg_{i}} \bullet \transfa_{ij}} \\
\functf'_{ij} \circ \functg_{j} \circ \functf_{j}
\ar[u]^{\transfb_{ij} \bullet \id_{\functf_{j}}} &
\functf'_{ij} \circ \functf'_{j} \circ \functg \ar[l]_{\id_{\functf'_{ij}} \bullet \transfb_{j}}
& \functf'_{i}\circ \functg \ar[ul]_{\transfb_{i}}
\ar[l]_{\transfa'_{ij} \bullet \id_{\functg}}
}} \qquad
\text{commutes.}
\] 
The $\shr$-functor $\functg\colon\stkc\to\stkc'$ is unique up to a unique
invertible transformation.
\item[(iii)] 
Let $\functg\colon\stkc\to\stkc'$ be as above, and let
$\functg'\colon\stkc\to\stkc'$ be the $\shr$-functor associated with the descent datum
$\pair{\functg'_i}{\transfb'_{ij}}{ij\in I}$. Consider the descent datum
$\uniple{\transfd_{i}}{i\in I}$, that is, $\transfd_{i}\colon
\functg_{i}\to\functg'_{i}$ are transformations on $U_{i}$ such that
\begin{equation}
\label{eq:alpha3}
\vcenter{\xymatrix{ \functg_{i} \circ \functf_{ij}
\ar[d]_{\transfd_{i} \bullet \id_{\functf_{ij}}} & \functf'_{ij} \circ \functg_{j}
\ar[l]^{\transfb_{ij}} \ar[d]^{\id_{\functf'_{ij}} \bullet \transfd_{j}} \\
\functg'_{i} \circ \functf_{ij} & \functf'_{ij} \circ \functg'_{j}
\ar[l]_{\transfb'_{ij}} 
}} \qquad
\text{commutes.}
\end{equation}
Then, there exists a unique transformation $\transfd\colon
\functg\to\functg'$ such that $\transfd|_{U_i} = \transfd_i$.
\end{itemize}
\end{proposition}

\begin{remark}\label{rk:ref}
Let $\shv = \{V_i\}_{i\in J}$ be open cover of $X$ finer than $\shu$, and choose a refinement map 
$\rho \colon J \to I$ (that is, $V_i \subset U_{\rho(i)}$ for any $i\in J$).
Let $D = \triplet{\stkc_i}{\functf_{ij}}{\transfa_{ijk}}{ijk\in I}$ be a descent datum defined on $\shu$ and set
$$
\tilde\stkc_i = \stkc_{\rho(i)}|_{V_{i}},  \quad \tilde\functf_{ij} = \functf_{\rho(i)\rho(j)}|_{V_{ij}} , \quad \tilde\transfa_{ijk} = \transfa_{\rho(i)\rho(j)\rho(k)}|_{V_{ijk}}.
$$
Then $\opb \rho D= \triplet{\tilde\stkc_i}{\tilde\functf_{ij}}{\tilde\transfa_{ijk}}{ijk\in J}$ is a descent datum on $\shv$ which defines an $\shr$-stack $\shr$-equivalent to that associated to $D$. 

Let  $\rho' \colon J \to I$ be another choice of a refinement map and $\opb {\rho'} D= \triplet{\tilde\stkc'_i}{\tilde\functf'_{ij}}{\tilde\transfa'_{ijk}}{ijk\in J}$ the associated descent datum on $\shv$. Set
$$
\functg_i = \functf_{\rho'(i)\rho(i)}|_{V_{i}} , \quad \transfb_{ij} = \opb {(\transfa_{\rho'(i)\rho'(j)\rho(i)}|_{V_{ij}}\bullet \id_{\tilde\functf_{ij}})} \circ (\id_{\tilde\functf'_{ij}} \bullet  \transfa_{\rho'(j)\rho(i)\rho(j)}|_{V_{ij}}).
$$
Then $\pair{\functg_i}{\transfb_{ij}}{ij\in J}$ is a descent datum as in (ii) and, since the $\functg_i$ are equivalences, it defines an $\shr$-equivalence between the $\shr$-stacks associated to $\opb \rho D$ and $\opb {\rho'}  D$.
\end{remark}

\subsection{Algebroid cocycles}\label{se:acocy}

We give here a description of $\shr$-algebroids and $\shr$-functors between them 
in terms of $\shr$-algebras and $\shr$-algebra morphisms.

\medskip
Let $\stka$ be an $\shr$-algebroid on $X$. By definition, there exists
an open cover $\{U_i\}_{i\in I}$ of $X$ such that $\stka|_{U_i}$ is non-empty.
For $\obja_i\in\stka(U_i)$ and $\sha_i = \shEnd[\stka](\obja_i)$, there are
$\shr$-equivalences $\functf_i\colon\stka|_{U_i} \to \astk{\sha_i}$. Choose
quasi-inverses $\qinv{\functf_i}$ and invertible transformations $\id \to
\qinv{\functf_j}\circ \functf_j$. Set $\functf_{ij} := \functf_i\circ \qinv{\functf_j}
\colon \astk{\sha_j} \to \astk{\sha_i}$ on $U_{ij}$. On $U_{ijk}$ there are
invertible transformations $\transfa_{ijk} \colon \functf_{ik} \to
\functf_{ij}\circ \functf_{jk}$ induced by $\id \to \qinv{\functf_j}\circ \functf_j$. On
$U_{ijkl}$ one checks that the diagram \eqref{eq:alpha} commutes. By
Proposition~\ref{pr:patch}~(i), the data
$\triplet{\stka_i}{\functf_{ij}}{\transfa_{ijk}}{i,j,k\in I}$ are enough to
reconstruct $\stka$, in the sense that the stack obtained by glueing these
data is $\shr$-equivalent to $\stka$.

The $\shr$-equivalence $\functf_{ij}\colon \astk{\sha_j} \to \astk{\sha_i}$ on
$U_{ij}$ is locally induced by $\shr$-algebra isomorphisms. There thus exist
an open cover $\{U_{ij}^\alpha\}_{\alpha\in A}$ of $U_{ij}$ such that
$\functf_{ij} = \astk{(f_{ij}^\alpha)}$ on $U_{ij}^\alpha$ for $f_{ij}^\alpha
\colon \sha_j \to \sha_i$ isomorphisms of $\shr$-algebras. On triple
intersections $U_{ijk}^{\alpha\beta\gamma} = U_{ij}^\alpha \cap U_{ik}^\beta
\cap U_{jk}^\gamma$, the invertible transformations $\transfa_{ijk} \colon
\astk{(f_{ik}^\beta)} \to \astk{(f_{ij}^\alpha f_{jk}^\gamma)}$ are given by
invertible sections $a_{ijk}^{\alpha\beta\gamma} \in
\sha_i(U_{ijk}^{\alpha\beta\gamma})$ such that $f_{ij}^\alpha f_{jk}^\gamma =
\ad(a_{ijk}^{\alpha\beta\gamma}) f_{ik}^\beta$. 
(Recall that we set $\ad(a)(b) = aba^{-1}$.)
On quadruple intersections
$U_{ijkl}^{\alpha\beta\gamma\delta\epsilon\varphi} =
U_{ijk}^{\alpha\beta\gamma} \cap U_{ijl}^{\alpha\delta\epsilon} \cap
U_{ikl}^{\beta\delta\varphi} \cap U_{jkl}^{\gamma\epsilon\varphi}$, the
commutative diagram \eqref{eq:alpha} is equivalent to the equality
$a_{ijk}^{\alpha\beta\gamma} a_{ikl}^{\beta\delta\varphi} =
f_{ij}^\alpha(a_{jkl}^{\gamma\varepsilon\varphi})
a_{ijl}^{\alpha\delta\varepsilon}$.

One can treat in the same manner $\shr$-functors and transformations. We
summarize the results in the next proposition. However, as indices of
hypercovers are quite cumbersome, we will not write them explicitly anymore.
Instead, we will assume that
\begin{equation}
\label{eq:hypercofinal}
     \text{open covers of $X$ are cofinal among hypercovers.}
\end{equation}
This is the case, for example, of paracompact spaces.

\begin{proposition}
\label{pr:glue}
Assume \eqref{eq:hypercofinal}.
Let $\{U_i\}_{i\in I}$ be a sufficiently fine open cover of $X$
\begin{itemize}
\item[(i)] 
Any $\shr$-algebroid $\stka$ is reconstructed from a non-abelian cocycle
$\triplet{\sha_i}{f_{ij}}{a_{ijk}}{i,j,k\in I}$, that is, $\sha_i$ are
$\shr$-algebras on $U_i$, $f_{ij}\colon \sha_j|_{U_{ij}} \to \sha_i|_{U_{ij}}$
are $\shr$-algebra isomorphisms and $a_{ijk} \in \sha_i(U_{ijk})$ are
invertible sections,
such that
\begin{equation}
\label{eq:2nonabcoc}
\begin{cases}
f_{ij}f_{jk} = \ad(a_{ijk})f_{ik}, &\text{in }
\shHom[\shr\text-\stack{Alg}_X](\sha_k, \sha_i)(U_{ijk}),
\\ a_{ijk} a_{ikl} = f_{ij}(a_{jkl}) a_{ijl}, &\text{in }
\sha_i(U_{ijkl}).
\end{cases}
\end{equation}
\item[(ii)] 
Let $\stka$ be as above, and let $\stka'$ be an $\shr$-algebroid
constructed from the non-abelian cocycle $\triplet{\sha'_i}{f'_{ij}}{a'_{ijk}}{i,j,k\in I}$.
Any $\shr$-functor $\functg\colon\stka\to\stka'$ is reconstructed from
a non-abelian cocycle $\pair{g_{i}}{b_{ij}}{i,j\in I}$, that is, $g_{i}\colon \sha_i \to \sha'_i$
are $\shr$-algebra morphisms and $b_{ij} \in \sha'_i(U_{ij})$ are invertible
sections,
such that
\begin{equation}
\label{eq:1nonabcoc}
\begin{cases}
g_i f_{ij}=\ad (b_{ij})f'_{ij}g_j, &\text{in
}\shHom[\shr\text-\stack{Alg}_X](\sha_j, \sha'_i)(U_{ij}),
\\ g_i(a_{ijk})b_{ik} = b_{ij}f'_{ij}(b_{jk})a'_{ijk}, &\text{in }
\sha'_i(U_{ijk}).
\end{cases}
\end{equation}
\item[(iii)] 
Let $\functg\colon\stka\to\stka'$ be as above, and let
$\functg'\colon\stka\to\stka'$ be constructed from the non-abelian cocycle
$\pair{g'_{i}}{b'_{ij}}{i,j\in I}$. Any transformation of $\shr$-functors
$\transfd\colon \functg\to\functg'$ is reconstructed from a non-abelian cocycle
$\uniple{d_{i}}{i\in I}$, that is, $d_i\in \sha'_i(U_i)$ are sections such that
\begin{equation}
\label{eq:0nonabcoc}
d_i b_{ij} = b'_{ij} f'_{ij}(d_j), \quad\text{in }
\sha'_i(U_{ij}).
\end{equation}
\end{itemize}
\end{proposition}

In particular, two non-abelian cocycles
\[
\triplet{\sha_i}{f_{ij}}{a_{ijk}}{i,j,k\in I}, \qquad
\triplet{\sha'_i}{f'_{ij}}{a'_{ijk}}{i,j,k\in I},
\]
are associated to $\shr$-equivalent $\shr$-algebroids if and only if, up to refinements, there exists a non-abelian cocycle
$\pair{g_{i}}{b_{ij}}{i,j\in I}$ satisfying \eqref{eq:1nonabcoc} with $g_i$
isomorphisms of $\shr$-algebras.

Viceversa, let $\stka$ be as in Proposition~\ref{pr:glue}~(i). For $i,j\in I$
let $\sha'_i$ be $\shr|_{U_i}$-algebras, $g_{i}\colon\sha_i\to\sha'_i$
isomorphisms of $\shr$-algebras and $b_{ij}\in\sha'_i(U_{ij})$ invertible
sections. Then, the non-abelian cocycle
$\triplet{\sha'_i}{f'_{ij}}{a'_{ijk}}{i,j,k\in I}$
defined by \eqref{eq:1nonabcoc} is associated to an
$\shr$-algebroid $\shr$-equivalent to $\stka$.

\begin{remark}
Let $\triplet{\sha_i}{f_{ij}}{a_{ijk}}{i,j,k\in I}$ be
a non-abelian cocycle associated to an
$\shr$-algebroid $\stka$. Then \eqref{eq:2nonabcoc} implies the relations
\[
f_{ii} = \ad(a_{iii}), \quad 
a_{iij} = a_{iii}, \quad
a_{ijj} = f_{ij}(a_{jji}),
\quad\text{ for any }i,j\in I.
\] 
Setting $\sha'_i = \sha_i$, $g_{i} = \id_{\sha_i}$, and $b_{ij} =
a_{iji}$ in \eqref{eq:1nonabcoc}, we get
\[ 
f'_{ij} = \ad(a_{iji}^{-1})f_{ij}, \quad a'_{ijk} = f'_{ij}(a_{jkj}^{-1}a_{jki}).
\] 
Then, the non-abelian cocycle $\triplet{\sha'_i}{f'_{ij}}{a'_{ijk}}{i,j,k\in I}$ is associated to an
$\shr$-algebroid $\shr$-equivalent to $\stka$, and it
satisfies the relations
\[ 
f'_{ii} = \id_{\sha_i}, \quad a'_{iij} = a'_{ijj} = 1,
\] 
of a normalized cocycle in the sense of~\cite{Bre94}. 
\end{remark}

\subsection{Module cocycles}\label{se:mcocy}

Let $\stka$ be the $\shr$-algebroid described over the open cover
$\{U_i\}_{i\in I}$ of $X$ by the non-abelian cocycle
$\triplet{\sha_i}{f_{ij}}{a_{ijk}}{i,j,k\in I}$. The stack of (left)
$\stka$-modules $\stkMod(\stka)$ is then described as in
Proposition~\ref{pr:patch}~(i) by the family
\[
\triplet{\stkMod(\sha_i)}{\stkMod(\astk{f_{ji}})}{\stkMod(a_{kji})}{i,j,k\in
I}
\] 
(note the inversion of indices due to the fact that $\stkMod(\dummy)$ is
contravariant).
By Morita theory, the functor $\stkMod(\astk{f_{ji}})={}_{f_{ji}}(\dummy)$ is isomorphic to $\shp_{ij}\tens[\sha_j](\dummy)$ for the invertible $\sha_i\tens[\shr]\sha_j^\op$-module $\shp_{ij}= {}_{f_{ji}}\sha_j$. We thus recover the description of twisted sheaves given in \cite{Kas08} (see also~\cite{Kas96,DP04,DS07}).

\begin{proposition}
Let $\stka$ be as above. An object of $\catMod(\stka)$ is described by a
family $\pair{\shm_i}{\varphi_{ij}}{i,j \in I}$, where
$\shm_i\in\catMod(\sha_i)$, and $\varphi_{ij}\in \Hom[\sha_i](
{}_{f_{ji}}\shm_j\vert_{U_{ij}} , \shm_{i} \vert_{U_{ij}})$ are isomorphisms,
such that for any $u\in\shm_k$ one has
\[ 
\varphi_{ij} ( \varphi_{jk} (u)) =
\varphi_{ik}(a_{kji}^{-1} u).
\]
\end{proposition}

\begin{proof}
Let $\stkc$ be an $\shr$-stack as in Proposition~\ref{pr:patch}~(i). The
statement follows by noticing that objects of $\stkc(X)$ are described by data
\[
\pair{\obja_i}{\morpha_{ij}}{i,j\in I},
\] 
where $\obja_i\in\stkc_i(U_i)$, and $\morpha_{ij}\colon
\functf_{ij}(\obja_j)\to \obja_i$ are isomorphisms in $\stkc_i(U_{ij})$, such
that
\[ 
\morpha_{ij} \circ \functf_{ij}(\morpha_{jk}) = \morpha_{ik} \circ
\transfa_{ijk}^{-1}(\obja_k)
\] 
as isomorphisms $\functf_{ij}\functf_{jk}(\obja_k) \isoto \obja_i$ in
$\stkc_i(U_{ijk})$.
\end{proof}

\end{document}